\documentclass{amsart}
\usepackage[top=2.5cm,left=2cm,right=2cm,bottom=3cm]{geometry}
\usepackage[toc,page]{appendix}
\usepackage{amsmath, amssymb, amstext,amsthm}
\usepackage{enumitem}
\usepackage{stackengine} 
\usepackage{bm} % shorter way to write bold in math mode
\usepackage{hyperref}
\usepackage[british]{babel}
\usepackage{verbatim}
\usepackage{esint}
\usepackage[toc]{appendix}
\usepackage{amsfonts}\usepackage{amscd}\usepackage{url}\usepackage{amsopn}% provides \DeclareMathOperator
\usepackage{bbm}% additional fonts
\usepackage{stmaryrd}
\usepackage{color}
\usepackage{xcolor}
\usepackage{csquotes}
\usepackage[%
    style=numeric, % numeric, alphabetic, authoryear, ect.
    isbn=false,
    doi=false,
    url=false,
    maxbibnames=99,
    backend=biber,]{biblatex}
\renewbibmacro{in:}{}

\DeclareFieldFormat[article,incollection]{title}{\mkbibitalic{#1}}
\DeclareFieldFormat[incollection]{booktitle}{#1}
\DeclareFieldFormat{journaltitle}{#1}
\DeclareFieldFormat[article]{volume}{#1}
\makeatletter

\newcounter{thm}
\numberwithin{thm}{section}
\newtheorem{remark}[thm]{Remark}
\newtheorem{definition}[thm]{Definition}
\newtheorem{theorem}[thm]{Theorem}
\newtheorem{lemma}[thm]{Lemma}
\newtheorem{proposition}[thm]{Proposition}

\newcommand{\E}{\mathbb{E}}

\renewcommand{\O}{\Omega}

\newcommand{\mcR}{\mathcal{R}}

\newcommand{\mbT}{\mathbb{T}}

\newcommand{\R}{\mathbb{R}}

\newcommand{\N}{\mathbb{N}}
\newcommand{\Z}{\mathbb{Z}}

\newcommand{\D}{\Delta}

\newcommand{\esssup}{\mathrm{ess~sup}}
\DeclareMathOperator*{\supp}{supp}

\DeclareMathOperator*{\divv}{div}
\DeclareMathOperator*{\Tr}{Tr}
\DeclareMathOperator*{\curl}{curl}
\newcommand{\Id}{\mathrm{Id}}

\begin{document}

\title{Nonuniqueness in Law for Stochastic Hypodissipative Navier--Stokes Equations}
\author{Marco Rehmeier, Andre Schenke}
\email{mrehmeier@math.uni-bielefeld.de, aschenke@math.uni-bielefeld.de}
\address{Fakult\"at f\"ur Mathematik, Universit\"at Bielefeld, 33615 Bielefeld, Germany}

\begin{abstract}
We study the incompressible hypodissipative Navier--Stokes equations with dissipation exponent $0 < \alpha < \frac{1}{2}$ on the three-dimensional torus perturbed by an additive Wiener noise term and prove the existence of an initial condition for which distinct probabilistic weak solutions exist. To this end, we employ convex integration methods to construct a pathwise probabilistically strong solution, which violates a pathwise energy inequality up to a suitable stopping time. This paper seems to be the first in which such solutions are constructed via Beltrami waves instead of intermittent jets or flows in a stochastic setting.
\end{abstract}

\maketitle

\section{Introduction}
In this work we are concerned with nonuniqueness for the stochastic hypodissipative (fractional) Navier--Stokes equations on the three-dimensional torus $\mathbb{T}^3$
\begin{equation}\label{HSNSE_Intro}
\begin{cases}\tag{$\text{HNSE}_{\text{st}}$}
&\partial_tv+\divv (v \otimes v)+\nabla p + (-\Delta)^{\alpha}v 
= dB,\\ 
&\text{div}(v) = 0,
\end{cases}
\end{equation}
in which, in contrast to the deterministic case, the velocity field $v$ is additionally perturbed by a Brownian motion $B$. The dissipation is fractional in the sense that it assumes the form of a fractional power $(-\D)^{\alpha}$ of the Laplacian, and hypodissipative in the sense that $\alpha \in (0,\frac{1}{2})$. In fact, we prove nonuniqueness in the class of probabilistically weak solutions (or martingale solutions) under suitable assumptions on the stochastic perturbation, as formulated in the main result, Theorem \ref{Main_Result}. In particular, our result rules out an application of the well-known Engelbert-Cherny-theory (cf. \cite{Engelbert91, Cherny01, Rehmeier21}) to prove pathwise uniqueness of \eqref{HSNSE_Intro} via the construction of probabilistic strong solutions and weak uniqueness, at least for the comparably large class of general weak solutions. For subclasses of solutions with certain energy constraints, this approach remains open.

To derive the nonuniqueness result, we construct a martingale solution to \eqref{HSNSE_Intro}, which violates an energy inequality and compare it to a second solution obtained by a classical Galerkin approximation argument that satisfies an energy inequality. Roughly speaking, the construction of the former solution consists of two independent steps: On the one hand, we construct a pathwise solution $u$ to \eqref{HSNSE_Intro} up to a bounded stopping time $\tau_L$, see Theorem \ref{Theorem_analySol}. This solution is analytically weak, but probabilistically strong, which is crucial in order to show that its law $P$ induces a martingale solution on $[0,\tau_L]$. This construction employs the so-called method of convex integration. On the other hand, we extend $P$ beyond $\tau_L$ to a solution on $[0,\infty)$. To this end, we make use of measure-theoretic and probabilistic arguments in order to concatenate the solution $P$ with a family of martingale solutions $R$ at the random time $\tau_L$ to obtain a solution $P \otimes _{\tau_L}R$, which coincides with $P$ up to $\tau_L$. The parameter $L>1$ will be chosen sufficiently large and accounts for the necessary violation of the energy inequality, see (\ref{Failure_energy_ineq}).

The pioneering idea to combine these methods led to the beautiful work \cite{HZZ_stoNSE} by Hofmanov\'{a}, Zhu and Zhu, where the seminal result of ill-posedness in law for martingale solutions to the classical stochastic Navier-Stokes equations is obtained. While \cite{HZZ_stoNSE} is one of the first works in which such techniques are employed to tackle nonuniqueness results for SPDEs, the history of convex integration techniques in the realm of fluid dynamics is much longer.

In the deterministic setting, the origins of convex integration date back at least to the work of Nash \cite{Nash54} and Kuiper \cite{Kuiper55} on $C^{1}$-isometric embeddings of Riemannian manifolds. Gromov later realized that the methods of Nash and Kuiper could be understood as instances of a more general principle called the $h$-principle, cf. \cite{Gromov86}. Building on the work of M\"uller and \v{S}ver\'{a}k \cite{MS03} and Kirchheim \cite{Kirchheim03}, De Lellis and Sz\'{e}kelyhidi \cite{DLS09} proved nonuniqueness of weak solutions to the Euler equations in $L^{\infty}$. Their work gave also new proofs of important results of Scheffer \cite{Scheffer93} and Shnirelman \cite{Shnirelman97, Shnirelman00}. However, it relied on Tartar's plane-wave analysis \cite{Tartar79} which prevented the authors from constructing continuous solutions. A breakthrough came when De Lellis and Sz\'{e}kelyhidi in \cite{DLS13} for the first time used genuinely multi-dimensional building blocks, so-called Beltrami waves, which were introduced more than two decades earlier by Constantin and Majda \cite{CM88}. Using such rapidly oscillating multi-dimensional waves enabled them to construct infinitely many solutions in $C^{1/10}$ in \cite{DLS14}, which turned out as a major step towards Onsager's conjecture from 1949 \cite{Onsager49}, one of the central longstanding unresolved conjectures in fluid dynamics. Roughly, it asserted that the Euler equations should be ill-posed in any $C^{\gamma}, \gamma < 1/3$. The seminal work \cite{DLS13} launched a race towards its rigorous proof, which proceeded along several important partial results (cf. \cite{Buckmaster14, Buckmaster15, BDLIS15, BDLS13, BDLS16, Isett13, IO16}), and was finally resolved by Isett \cite{Isett18} in the nonconservative case and by Buckmaster, De Lellis, Sz\'{e}kelyhidi and Vicol \cite{BDLSV18} in the dissipative case. Among further refinements of the original methods, crucial improvements were achieved by consideration of a new type of building blocks (Mikado flows, introduced by Daneri and Sz\'{e}kelyhidi in \cite{DS17}) and Isett's gluing technique. Shortly thereafter, Buckmaster and Vicol \cite{BV19a} introduced yet another type of building block (intermittent Beltrami flows) to prove nonuniqueness of weak solutions to the 3D Navier--Stokes equations. A profound overview of these developments can be found in \cite{BV19b}.

In the light of this extensive list of ill-posedness results, it is natural to ask whether the situation is different in the presence of stochastic forces. Indeed, as a well-known general rule, under appropriate conditions, PDEs may be regularized by introducing random forces acting on the equation in the sense that certain ill-posed deterministic problems have well-posed stochastic counterparts. In a nutshell, this phenomenon of \textit{regularization by noise} (cf. \cite{Veretennikov80, KR05, FGP10}, and \cite{Gess18} for a nice overview of the subject) works when a sufficiently active noise (e.g. acting nontrivially in sufficiently many directions) moves solutions away from singularities of the underlying vector field. In the light of such results, the above mentioned ill-posedness results obtained by convex integration techniques did not rule out the possibility of well-posedness of their stochastic counterparts. However, as mentioned above, the general hopes that (pathwise) uniqueness might hold for prominent fluid dynamical stochastic PDEs were shattered by \cite{HZZ_stoNSE}.

The present paper seeks to apply the ideas of \cite{HZZ_stoNSE} to the case of the stochastic fractional Navier--Stokes equations in the hypodissipative case. This equation was introduced by J.-L. Lions \cite{Lions69} in 1969, who proved the existence of strong solutions if $\alpha \geq \frac{5}{4}$, i.e. for a dissipation effect stronger than in the case of the classical Navier--Stokes equations (the hyperdissipative case). Later, Tao \cite{Tao09} proved that in the critical case $\alpha = \frac{5}{4}$, existence of a global regular solution still holds even when the equation is perturbed by a logarithmically supercritical operator. Partial regularity results in the flavor of Caffarelli, Kohn and Nirenberg \cite{CKN82} were provided in the works \cite{CDLM20, RWW16, TY15}. In the stochastic case, the only well-posedness result dealing explicitly with the fractional Navier--Stokes equations seems to be due to Debbi \cite{Debbi16}, who shows existence of local mild and martingale solutions as well as Beale-Kato-Majda-type criteria for uniqueness.

Unsurprisingly, the fractional Navier--Stokes equations have also been treated by convex integration techniques, both in the deterministic and the stochastic case, in two as well as in three spatial dimensions. In three dimensions, the first works in this direction are due to Colombo, De Lellis and De Rosa \cite{CDLDR18} who proved ill-posedness even in the class of Leray solutions for the hypodissipative Navier--Stokes equations with $0 < \alpha < \frac{1}{5}$. De Rosa \cite{DR19} extended this result to an Onsager-type theorem. The hyperdissipative deterministic case $\alpha \in [1, 5/4)$ was treated by Luo and Titi \cite{LT20} as well as Buckmaster, Colombo and Vicol \cite{BCV22}. Recently, Yamazaki proved nonuniqueness in law to the hyperdissipative stochastic Navier--Stokes equations. The two-dimensional deterministic case was treated by Luo and Qu \cite{LQ20}, and a stochastic result is due to Yamazaki \cite{Yamazaki20b}.

We would also like to mention the works by Chiodaroli, Feireisl and Flandoli \cite{CFF19} as well as Breit, Feireisl and Hofmanov\'{a} \cite{BFH20} on the stochastic compressible case. Furthermore we would like to mention the more recent works of Hofmanov\'{a}, Zhu and Zhu on Markov selections for stochastic Navier--Stokes equations \cite{HZZ21a} and stochastic Navier--Stokes equations with space-time white noise \cite{HZZ21b} as well as Yamazaki on stochastic Boussinesq equations \cite{Yamazaki21a} and stochastic MHD equations \cite{Yamazaki21c}.

\subsection*{Main Result} Amid these ill-posedness results, the case of the incompressible 3D stochastic hypodissipative Navier--Stokes equations for $\alpha \in (0,\frac{1}{2})$ remains open. The purpose of this paper is twofold: We treat this case by showing that martingale solutions to \eqref{HSNSE_Intro} are not unique in law, provided the stochastic forces are modeled by a $GG^*$-Wiener process for a suitable trace-class operator $GG^*$ on $L^2_{\sigma}$, i.e. more precisely we prove the following result\footnote{ Note that the independent work of Yamazaki \cite{Yamazaki21b} that was published at the same time as this article contains the same result.}.
\begin{theorem}\label{Main_Result}
	Let $\alpha \in (0,\frac{1}{2})$ and assume that for some $\sigma >0$ we have $\Tr \left[ A_{\alpha}^{\rho_{0}} G G^{*} \right] = \Tr \left[ (-\D)^{\rho_{0} \alpha} G G^{*} \right] < \infty$ for $\rho_{0} = \frac{5 + 2\sigma - 2 \alpha}{2 \alpha}$. Then
	for any $T >0$, there exist two martingale solutions to \eqref{HSNSE_Intro} on $[0,+\infty)$ with a common deterministic initial condition $x_0 \in L^2_{\sigma}$ which are distinct on $[0,T]$.
\end{theorem}

In the race towards the proof of Onsager's conjecture, a crucial step was to deal with the so-called \textit{transport error} of the Reynolds stress. This step is simplified in the Navier--Stokes equations thanks to the presence of the ``full'' Laplacian in the equation (which, however, raises other delicate issues when trying to achieve intermittency). The aforementioned recent papers of Hofmanov\'{a} \emph{et al.} and Yamazaki \cite{HZZ_stoNSE, Yamazaki20a} on convex integration in the stochastic setting mostly deal with a Navier--Stokes or Navier--Stokes-like setting. Here, we will be concerned with the hypodissipative Navier--Stokes equations for powers $0 <\alpha <\frac{1}{2}$, which is closer in spirit to the Euler equations. We expect that more careful handling of the transport error will also be crucial in order to approach a potential stochastic analog of Onsager's conjecture. This work may be considered as a first humble step into this direction. 

We study in a way the technically simplest case (Beltrami waves) to isolate the problems related to the transport as clearly as possible. To the best of our knowledge, this is the first paper to use Beltrami waves in a stochastic setting. The interplay of the phase transport with the stopping times needed for the stochastic case gives rise to a few technical issues, which we address in this work: In comparison to the deterministic case, it turns out that the right vector field along which to transport the phase is $v_{\ell}+z_{\ell}$ instead of $v_{\ell}$ (see Section \ref{Sect_Mollification}). This transport causes the constant $C_L$ defined in \eqref{def_C-L}, which cannot be absorbed into the small constant $c_R$ (see \eqref{cond_on_L}). In principle this constant could grow exponentially in the iteration $q \to q+1$, thereby hampering the convergence of the iteratively constructed approximate solutions $v_q$. However, we show that this is not the case.

Note that the noise in our main result might not be ``sufficiently'' active as we assume a decay on the modes of the noise. This leaves the possibility that uniqueness in law still holds for rougher noises.

%******************************************************************************
\subsection*{Organization of the paper}\label{STMHD_ssec_results}
The rest of this paper is organized as follows. Section 2 establishes the notation of the paper and contains the definition of weak solutions (martingale solutions) we aim to study. Moreover, we present the necessary measure theoretic preparations and decompose (\ref{HSNSE_Intro} into a linear stochastic and nonlinear deterministic part. Section 3 is devoted to the convex integration methods in order to construct the analytically weak solution to the nonlinear deterministic equation. In Section 4, we use this construction and all other preparations to finally prove Theorem \ref{Main_Result}. The appendices contain the proofs of \ref{Prop_Reg_z_maintext} (\ref{AppA}) and Proposition \ref{Thm3.1.-analog} (\ref{AppB}) as well as information on Beltrami waves (\ref{App_Beltrami_waves}) and a technical prerequisite for the proofs in Section 3 (\ref{AppD}).

%******************************************************************************

\section{Preliminaries}

\subsection{Notation}
We denote by $\mathbb{N}_0 := \{0,1,2,\dots\}$ the natural numbers including $0$.
For $x,y \in \mathbb{R}^d$, the standard Euclidean norm and inner product are denoted by $|x|$ and $x\cdot y$, respectively. $G^*$ denotes the adjoint of an operator $G$ and the space of Hilbert-Schmidt operators between Banach spaces $X$ and $Y$ is denoted by $L_2(X,Y)$. 
\subsubsection*{Periodic functions, Sobolev spaces and fractional Laplacian}
It is standard to identify $2\pi$-periodic functions $f: \mathbb{R}^3 \to \mathbb{C}$ with functions on the torus $\mathbb{T}^3 := S^1\times S^1 \times S^1$, where $S^1 := \{e^{i\theta}, \theta \in \mathbb{R}\}$. Therefore, denote by $C(\mathbb{T}^3, \mathbb{C}^d)$ the space of continuous $2\pi$-periodic functions $f: \mathbb{R}^3 \to \mathbb{C}^d$ and, for $k \in \mathbb{N} \cup \{\infty\}$, by $C^k(\mathbb{T}^3,\mathbb{C}^d)$ those $f$ which are $k$ times continuously differentiable. In the case $d =3$, we simply write $C^k$. Similarly, we consider $C^{\gamma}$, the space of $2\pi$-periodic H\"older-continuous functions $f: \mathbb{R}^3 \to \mathbb{C}^3$ of order $0<\gamma <1$.\\
For $p \in [1,\infty]$, let $L^p(\mathbb{T}^3, \mathbb{C})$ be the set of $dx$-equivalence classes of $p$-integrable functions $f: \mathbb{T}^3 \to \mathbb{C}$ on $[-\pi,\pi]^3$ and let $L^p(\mathbb{T}^3,\mathbb{C}^d)$ denote the space of vector fields $f = (f_1,\dots,f_d)$ with $f_i \in L^p(\mathbb{T}^3,\mathbb{C})$, which we abbreviate by $L^p$ for $d =3$. The spaces $L^p(\mathbb{T}^3,\mathbb{C}^d)$ are Banach spaces with norm $||f||^p_{L^p} := \frac{1}{(2\pi)^3}\int_{[-\pi,\pi]^3}|f|^pdx$ for $p <\infty$ and $||f||_{L^{\infty}} := \esssup_{x \in [-\pi,\pi]^3}|f(x)|$ for $p =\infty$, and Hilbert spaces for $p=2$ with scalar product $\langle f,g \rangle_{L^2}:= \frac{1}{(2\pi)^3}\int_{[-\pi,\pi]^3}f\cdot\bar{g}dx$. The symbols $||\cdot||_{L^p}$ and $\langle \cdot, \cdot \rangle_{L^2}$ are used without stressing the dimension of the target space. Furthermore, we write $L^{2}_{0} := \{f \in L^2: \int_{\mathbb{T}^{3}} f dx = 0 \}$ for the closed subspace of $3$D-vector fields with zero mean.\\
The functions $ [-\pi,\pi]^3 \ni x \mapsto e^{ik\cdot x}$ with $k \in \mathbb{Z}^3$ form an orthonormal basis in $L^2(\mathbb{T}^3,\mathbb{C})$. Therefore it holds that $f = \sum_{k \in \mathbb{Z}^3} \hat{f}_k e^{ik\cdot x}$ for $f \in L^2(\mathbb{T}^3,\mathbb{C})$, where $\hat{f}_k := \langle f, e^{ik\cdot x} \rangle_{L^2}$ denotes the \textit{$k$-th Fourier coefficient} of $f$ and the above series is the \textit{Fourier series} of $f$. The map $\mathcal{F}: L^2(\mathbb{R}^3,\mathbb{C}) \mapsto \ell^2(\mathbb{Z}^3,\mathbb{C})$, $\mathcal{F}: f \mapsto (\hat{f}_k)_{k \in \mathbb{Z}^3}$ is an isometric isomorphism of Hilbert spaces. For $L^2$, the above considerations apply component-wise.\\
Since we study incompressible equations, we introduce the closed subspace $L^2_{\sigma}$ := $\{f \in L_0^2: \text{div}(f)=0\}$, where for a generic element $f \in L_2$, $\text{div}(f)$ is understood in distributional sense. $L^2_{\sigma}$ inherits the Hilbert space-structure from $L^2$ and we denote by $\mathbb{P}: L^2 \to L^2_{\sigma}$ the usual orthogonal projection.\\
For $s \geq 0$, we consider the solenoidal (fractional) Sobolev spaces $$H^s:= H^s(\mathbb{T}^3,\mathbb{C}^3) := \bigg\{f \in L^2_{\sigma}: ||(1-\Delta)^{s/2}f||_{L^2} < \infty\bigg\} = \bigg\{f=(f_1,f_2,f_3) \in L^2_{\sigma}: \sum_{k \in \mathbb{Z}^3}(1+|k|^2)^s\hat{(f_i)}_k^2 < \infty,\, i \leq 3\bigg\},$$
which are $\mathbb{C}$-Hilbert spaces with scalar product
$$\langle f, g \rangle_{H^s} := \sum_{i\leq3}\langle (1-\Delta)^{s/2}f_i, \overline{(1-\Delta)^{s/2}g_i}\rangle_{L^2} = \sum_{i\leq3}\sum_{k \in \mathbb{Z}^3}(1+|k|^2)^s\hat{(f_i)}_k \overline{\hat{(g_i)}}_k,$$
where the equalities follow from the fact that the symbol of $(1-\Delta)^{s/2}$ as a Fourier multiplier is $(1+|k|^2)^{s/2}$. The induced norm is denoted by $||f||_{H^s}$. For $s \in \mathbb{N}$, $H^s$ coincides with the classical Sobolev space of functions with square-integrable weak derivatives up to order $s$.\\
For $s >0$ we define $H^{-s}$ as the dual of $H^s$ with the standard dual norm. For $s \geq 0$, we denote by $\langle \cdot, \cdot \rangle_{(-s,s)}: H^{-s}\times H^s \to \mathbb{C}$ the dual pairing $\langle f, g \rangle_{(-s,s)} := f(g)$.
We recall that for any $-\infty < r < s < +\infty$, we have the compact embedding $H^s \subseteq H^r$, see \cite[Eq. $(3.12)$, p. 330]{Taylor1}).\\
For $\alpha \in (0,1)$, the fractional Laplace operator $-(-\Delta)^\alpha$ is the operator with symbol $|k|^{2\alpha}$ as a Fourier multiplier, i.e. for any $f \in H^s$, $s \in \mathbb{R}$, it has the (formal) Fourier series
$$(-\Delta)^{\alpha}f(x) = \sum_{k \in \mathbb{Z}^3}|k|^{2\alpha}\hat{f}_k,$$
which is convergent if and only if $f \in H^{2\alpha}$. In particular, $(-\Delta)^{\alpha}: H^s \to H^{s-2\alpha}$, $f \mapsto (-\Delta)^{\alpha}f$ is continuous for $s \in \mathbb{R}$. We recall the following relation between $(-\Delta)^{s/2}$ and the norm $||\cdot||_{H^s}$: For $s \in (0,1)$ there is a universal constant $C=C(s)>1$ such that
$$C^{-1}\bigg(||f||^2_{L^2}+||(-\Delta)^{s/2}f||^2_{L^2}\bigg) \leq ||f||^2_{H^s}\leq C\bigg(||f||^2_{L^2}+||(-\Delta)^{s/2}f||^2_{L^2}\bigg),\quad f \in H^s.$$

\subsubsection*{Function spaces}
For $N \in \mathbb{N}_0 \cup \{\infty\}$, $0<\gamma <1$, $g \in C^{\gamma}$, $f \in C^{N}$ and $h \in C^{N+1}$ as above, we introduce the following standard (semi-)norms. 
\begin{align*}
[g]_{C^\gamma} &:= \sup_{x\neq y \in \mathbb{R}^3}|\frac{|g(x)-g(y)|}{|x-y|^{\gamma}},\quad
||g||_{C^\gamma} := [g]_{C^{\gamma}}+\sup_{x \in \mathbb{R}^3}|g(x)|,\\
[f]_{C^N} &:= \max_{|\alpha| = N}\sup_{x \in \mathbb{R}^3}|D^{\alpha}f(x)|,\quad
||f||_{C^N} := \sum_{k = 0}^N[f]_{C^N},\\
[h]_{N+\gamma} &:= \max_{|\alpha| = N}[D^{\alpha}h]_{C^\gamma}
\end{align*} where for a multiindex $\alpha \in \mathbb{N}^3$, $D^{\alpha}$ denotes the corresponding partial derivative.
For a time interval $I \subseteq \mathbb{R}_+ \cup \{\infty\}$ and a Banach space $X$, we write $C(I,X)$ for the space of continuous functions $f: I \to X$ equipped with the topology of locally uniform convergence. For vector fields $v: I \times \mathbb{R}^3 \to \mathbb{R}^3$ such that $v(t,\cdot)$ is $2\pi$-periodic for each $t \in I$ we use the notation
\begin{align*}
||v||_{C^0_IC^N_x} := \sup_{t \in I}||v(t,\cdot)||_{C^N},
\end{align*}
which we abbreviate as $||v||_{C^0_tC^N_x}$ and $||v||_{C^0_{I,x}}$ for the special cases $I = [0,t]$ and $N =0$, respectively. Similarly, for $N \in \mathbb{N}_0$, we write
\begin{align*}
||v||_{C^N_{I,x}} := \max_{k+|\alpha| \leq N}||\partial_t^kD^{\alpha}v(t,x)||_{C^0_{I,x}},
\end{align*}
again abbreviated as $C^N_{t,x}$ in the case $I = [0,t]$.

For $I \subseteq \mathbb{R}_+ \cup \{\infty\}$, a Banach space $X$ and $p \in [1,\infty]$, we denote by $L^p(I,X)$ [$L_{loc}^p(I,X)$] the usual space of equivalence classes of [locally] $p$-integrable functions $f: I \to X$ with the usual norm.

\subsubsection*{Measure theoretic and probabilistic elements} We consider the path space $\Omega_0 := C(\mathbb{R}_+,H^{-3})$ with the topology of locally uniform convergence, which renders $\Omega_0$ a Polish space with Borel-$\sigma$-algebra $\mathcal{B} = \sigma(\pi_t, t \geq 0)$, where $\pi_t : x \mapsto x(t) \in H^{-3}$ denotes the canonical projection. A generic element in $\Omega_0$ is $x =(x_t)_{t \geq 0} \in \Omega_0$. Let $(\mathcal{B}_t^0)_{t \geq 0}$ denote the filtration with $\mathcal{B}_t^0 := \sigma(\pi_s, s \leq t)$ and define $(\mathcal{B}_t)_{t \geq 0}$ as the corresponding right-continuous filtration. Furthermore, we set $\mathcal{B}^t := \sigma(\pi_r,r \geq t)$.\\
The quadratic variation process of a real-valued process $M$ on a probability space is denoted $t \mapsto \langle \langle M \rangle \rangle_t$. For a Polish space $X$ with Borel $\sigma-$algebra $\mathcal{B}(X)$, we write $\mathcal{P}(X)$ for the space of probability measures on $\mathcal{B}(X)$ endowed with the topology of weak convergence of measures. With this topology, $\mathcal{P}(X)$ is metrizable, complete and separable if so is $X$.

\subsection{Setting and Framework}\label{Sect_setting}
Let $\alpha \in (0,\frac{1}{2})$. We consider the stochastic 3D hypodissipative incompressible Navier--Stokes equation \eqref{HSNSE_Intro} on $\mathbb{T}^3$, where $B$ is a $GG^*$-Wiener process on a prescribed probability space $(\Omega,\mathcal{F},\mathbb{\mathbf{P}})$ for a Hilbert-Schmidt operator $G \in L_2(U,L^2_{\sigma})$ on some Hilbert space $U$.

For abbreviation, we set
 $$F_{\alpha}: L^2_{\sigma} \to H^{-3},\quad F_{\alpha}(y):= -\text{div}(y \otimes y)-(-\Delta)^{\alpha}y.$$
 That indeed $F_{\alpha}$ maps to $H^{-3}$ follows from Lemma \ref{Lem_ext_F-alpha}. The constant $C$ appearing in its proof is due to the Sobolev inequality $||\nabla y||_{L^{\infty}} \leq C ||y||_3$, which is one main reason for the choice of $H^{-3}$ in the definition of $\Omega_0$. Indeed, $l=3$ is the smallest integer such that the Sobolev embedding $H^{l-1} \hookrightarrow L^{\infty}$ holds.

We introduce the notion of martingale solutions to \eqref{HSNSE_Intro}, which we consider throughout this paper. 
\begin{definition}\label{DefMgSol}
Let $\gamma \in (0,1)$ and $(s, x_0) \in \mathbb{R}_+\times L^2_{\sigma}$. A probability measure $P \in \mathcal{P}(\Omega_0)$ is a martingale solution to \eqref{HSNSE_Intro} on $[s,\infty)$ with initial condition $(s,x_0)$, if
\begin{enumerate}[label = (M\arabic*)]
	\item\label{M1} $P\big(x \in \Omega_0: x(t) = x_0 \, \text{ for all } \, t \in [0,s]\big) = 1$.
	\item\label{M2} For each $e \in H^3$ the process on $\Omega_0$
	\begin{equation}\label{DefMgM2}
	M^e_s(t,x) := \langle x(t)-x_0,e \rangle_{(-3,3)} - \int_s^t \langle F_{\alpha}\big(x(r)\big),e \rangle_{(-3,3)}dr
	\end{equation}
	is a continuous real-valued, square-integrable $(\mathcal{B}_t)_{t \geq s}$-martingale with respect to $P$ with quadratic variation
	\begin{equation}
	\langle\langle M^e_s \rangle\rangle_t(x) = (t-s)||G^*e||^2_U, \quad  t\geq 0 \quad \text{ for }P-\text{a.e. }x \in \Omega_0.
	\end{equation}
	\item\label{M3} For each $q \in \mathbb{N}$ there is a nonnegative continuous function $t \mapsto C_{t,q} = C_{t,q}(s,x_0,P)$ such that for every $t \geq s$
	\begin{equation}\label{M3-ineq}
	\mathbb{E}_P\bigg[\underset{u \in [0,t]}{\sup}||x(u)||^{2q}_{L^2}+ \int_s^t ||x(u)||^{2(q-1)}_{L^2}||x(u)||^2_{H^\gamma} du\bigg] \leq C_{t,q}\big(||x_0||^{2q}_{L^2}+1\big).
	\end{equation}
\end{enumerate}
\end{definition}
\begin{remark}\label{Rem_Def_MGsol}
	\begin{enumerate}
		\item[(i)] By \eqref{M3-ineq} it follows that the complement of  $L^{\infty}_{\mathrm{loc}}(\mathbb{R}_+, L^2_{\sigma})$ in $C(\mathbb{R}_+,H^{-3})$ is a $P$-negligible set. Since $H^{-3} \hookrightarrow L^2_{\sigma}$ continuously, by \cite[Lemma 2.1]{FR08}, it follows that $P$ is concentrated on weakly continuous paths in $L^2_{\sigma}$ and that $L^{\infty}_{\mathrm{loc}}(\mathbb{R}_+, L^2_{\sigma}) \in \mathcal{B}$. In particular, $x(t) \in L^2_{\sigma}$ for all $t \geq 0$ $P$-a.s. Since the continuous embeddings $H^{3} \hookrightarrow L^2_{\sigma} \hookrightarrow H^{-3}$ are dense, this also implies $\langle x(\cdot),e\rangle_{(-3,3)} = \langle x(\cdot), e \rangle_{L^2}$ $P$-a.s. in \eqref{DefMgM2} and, in view of \eqref{eq_bounddness_F-alpha}, that the second term of the right-hand side of \eqref{DefMgM2} is well-defined.
		\item[(ii)] Since $G$ is Hilbert-Schmidt, $GG^*$ is symmetric, non-negative and of trace class in $L^2_{\sigma}$. By the assumption $\Tr \left[ (-\D)^{-\rho_{0} \alpha} G G^{*} \right] < \infty$, there is an orthonormal basis $\{e_j\}_{j \geq 1}$ of $L^2_{\sigma}$ in $H^3$ consisting of eigenvectors of $GG^*$. Denote by $\{\lambda_j\}_{j \geq 1}$ the corresponding sequence of eigenvalues with $\lambda_j >0$. By (M2), in the context of the above definition, $\lambda_j^{-1}M^{e_j}$ has quadratic variation
		$$ \langle \langle \lambda_j^{-1}M^{e_j} \rangle \rangle_t= t-s,$$
		i.e. $M^{e_j}$ is a real-valued $(\mathcal{B}_t)_{t \geq s}$-Brownian motion on $\Omega_0$ under $P$. Consequently, $\sum_{j \geq 1}M^{e_j}e_j$ is an $L^2_{\sigma}$-valued $GG^*$-Wiener process starting from $s$ on $(\Omega_0,\mathcal{B}, (\mathcal{B}_t)_{t \geq s},P)$.
	\end{enumerate}
\end{remark}
Similarly, we define martingale solutions up to a stopping time $\tau$. To this end, we introduce the space of paths stopped at $\tau$, i.e. $\Omega_{0,\tau} := \{x(\cdot \wedge \tau): x \in \Omega_0 \} = \{x \in \Omega_0: x = x(\cdot \wedge \tau)\}$ and note that $\Omega_{0,\tau} \in \mathcal{B}(\Omega_0)$ and hence $\mathcal{P}(\Omega_{0,\tau}) \subseteq \mathcal{P}(\Omega_0)$.
\begin{definition}\label{DefMgSol_stop}
	Let $\gamma \in (0,1)$, $(s, x_0) \in \mathbb{R}_+\times L^2_{\sigma}$ and $\tau$ be a $(\mathcal{B}_t)_{t \geq s}$-stopping time. A probability measure $P \in \mathcal{P}(\Omega_{0,\tau})$ is a martingale solution to \eqref{HSNSE_Intro} on $[s,\tau]$ with initial condition $(s,x_0)$, if
	\begin{enumerate}
		\item [(M1)] $P\big(x \in \Omega_0: x(t) = x_0 \, \forall \, t \in [0,s]\big) = 1$.
		\item[(M2)] For each $e \in H^3$ the process on $\Omega_0$
		\begin{equation}\label{DefMgM2_stop}
		M^e_s(t\wedge \tau,x) := \langle x(t\wedge \tau)-x_0,e \rangle_{(-3,3)} - \int_s^{t \wedge \tau} \langle F_{\alpha}\big(x(r)\big),e \rangle_{(-3,3)}dr
		\end{equation}
		is a continuous real-valued, square-integrable $(\mathcal{B}_t)_{t \geq s}$-martingale with respect to $P$ with quadratic variation
		\begin{equation}
		\langle\langle M^e_s \rangle\rangle_t(x) = (t\wedge \tau-s)||G^*e||^2_U, \quad t \geq 0 \quad \text{ for } P-\text{a.e. }x\in \Omega_0.
		\end{equation}
		\item[(M3)] For each $q \in \mathbb{N}$ there is a nonnegative continuous function $t \mapsto C_{t,q} = C_{t,q}(s,x_0,P,\tau)$ such that for every $t \geq s$
		\begin{equation}\label{M3-ineq_stop}
		\mathbb{E}_P\bigg[\underset{u \in [0,t\wedge \tau]}{\sup}||x(u)||^{2q}_{L^2}+ \int_s^{t\wedge \tau} ||x(u)||^{2(q-1)}_{L^2}||x(u)||^2_{H^\gamma} du\bigg] \leq C_{t,q}\big(||x_0||^{2q}_{L^2}+1\big).
		\end{equation}
	\end{enumerate}
\end{definition}
In this context, the contents of Remark \ref{Rem_Def_MGsol} hold accordingly.
\\

We shall need the following existence- and stability-result for martingale solutions to \eqref{HSNSE_Intro}. Below by $\mathcal{C}\big(s,x_0,(C_{t,q})_q\big)$ we denote the set of all martingale solutions with initial condition $(s,x_0)$, which fulfill the inequality \eqref{M3-ineq} of Definition \ref{DefMgSol} with respect to a common family $(C_{t,q})_{q \geq 1}$.
\begin{theorem}\label{Thm3.1.-analog}
\begin{enumerate}[label = (\roman*)]
	\item\label{Thm.3.1-analog_itm1} There is a family of continuous nonnegative functions $t \mapsto C_{t,q}$ such that for each $(s,x_0) \in \mathbb{R}_+\times L^2_{\sigma}$, there exists a martingale solution $P = P_{s,x_0} \in \mathcal{C}\big(s,x_0,(C_{t,q})_q\big)$ to \eqref{HSNSE_Intro} as in Definition \ref{DefMgSol} with the choice $\gamma = \alpha$.
	\item\label{Thm.3.1-analog_itm2} If for each $n \in \mathbb{N}$, $P_n \in \mathcal{C}\big(s_n,x_n,(\tilde{C}_{t,q})_q\big)$ and $(s_n,x_n) \underset{n \to \infty}{\longrightarrow}(s,x_0)$ in $\mathbb{R}_+\times L^2_{\sigma}$, then there is a subsequence $(P_{n_k})_{k \geq 1}$ which converges weakly in $\mathcal{P}(\Omega_0)$ to some $P \in \mathcal{C}\big(s,x_0,(\tilde{C}_{t,q})_q\big)$.
\end{enumerate}
\end{theorem}
Note that in the second part of the theorem the constants $\tilde{C}_{t,q}$ are assumed to be independent of $n$.\\
 The proof of the first part is based on the general existence result for martingale solutions via Galerkin approximations of \cite{GRZ_MgEx09}, while the second part is a close adaptation of the analogous result for the case $\alpha =1$, i.e. the classical stochastic Navier--Stokes equations, see \cite[Thm. 3.1]{HZZ_stoNSE}. For the sake of completeness, we present a proof in Appendix B.
\begin{remark}\label{Rem_simpleEnergyIneq_GalerkinSol}
	The construction of $P_{s,x_0}$ in the above theorem via a Galerkin approximation in \cite{GRZ_MgEx09} implies the energy inequality
\begin{equation}\label{simpleEnergyIneqLocal}
\mathbb{E}_{P_{s,x_0}}\big[||x(t)||^2_{L^2}\big] \leq ||x_0||^2_{L^2}+(t-s)\cdot\Tr(GG^*), \quad t \geq s.
\end{equation}
Indeed, the Galerkin approximations $(P_n)_{n \geq 1}$ to $P_{s,x_0}$ in the proof of \cite{GRZ_MgEx09} satisfy \eqref{simpleEnergyIneqLocal} uniformly in $n \geq 1$. Hence, the claim follows from the weak convergence $P_n \underset{n \to \infty}{\longrightarrow}P_{s,x_0}$ in $\mathcal{P}(\Omega_0)$, the fact that $x \mapsto ||x(t)||^2_{L^2}$ is lower semicontinuous on $\Omega_0$ and since $x_0 \in L^2_{\sigma}$ is deterministic.
\end{remark}

\subsection{Measure theoretic preliminaries}\label{Section_measure_prelims}
Concerning our main result Theorem \ref{Main_Result}, we aim to construct global martingale solutions on $[0,+\infty)$ with distinct laws. However, the convex integration methods of Section \ref{Section_CI} only allow to construct a solution $P$ up to a bounded stopping time $\tau_L \leq L$ for large $L>1$. To extend such $P$ to $[0,+\infty)$, we would like to make use of a classical measure theoretical extension technique, see \cite[Theorem 6.1.2.]{StroockVaradh2007}: At the random time $\tau_L$, we want to "glue together in a pathwise sense" $P$ with martingale solutions $R_{\tau_L(x),x(\tau_L(x))}$ starting at the pathwise end point $x(\tau_L(x)) \in L^2_{\sigma}$ at time $\tau_L(x)$. The existence of such a family $R = (R_{\tau_L(x),x(\tau_L(x))})_{x \in \Omega_0}$ is guaranteed by Theorem \ref{Thm3.1.-analog}. However, $\tau_L$ turns out to be a stopping time with respect to the right-continuous filtration $(\mathcal{B}_t)_{t \geq 0}$ only, which rules out an application of the classical results in \cite{StroockVaradh2007}. Instead, we use the following results of \cite{HZZ_stoNSE}, which are tailored to address this issue. We omit parts of the proofs, for which we refer to the excellent paper \cite{HZZ_stoNSE}.
\begin{proposition}\label{PropMeasurability1}
	Let $\tau$ be a bounded $(\mathcal{B}_t)_{t \geq 0}$-stopping time. Then for every $x \in \Omega_0$ there exists $Q_{x} \in \mathcal{P}(\Omega_0)$ such that
	\begin{equation}\label{Prop.Measurability_aux1}
	Q_{x}\big(x' \in \Omega_0: x(t)=x'(t) \text{ for }0 \leq t \leq \tau(x)\big) =1
	\end{equation}and
	\begin{equation}
	Q_x(A) = R_{\tau(x),x(\tau(x))}(A) \quad \text{ for all }A \in B^{\tau(x)},
	\end{equation}
	where $R_{\tau(x),x(\tau(x))} \in \mathcal{P}(\Omega_0)$ is a martingale solution to \eqref{HSNSE_Intro} with initial value $x(\tau(x)) \in L^2_{\sigma}$ at time $\tau(x)$. Furthermore, $x \mapsto Q_x(B)$ is $\mathcal{B}_{\tau}$-measurable for each $B \in \mathcal{B}$.
\end{proposition}

Below we aim to extend a martingale solution $P$ on $[0,\tau_L]$ to a martingale solution $P \otimes_{\tau_L} R \in \mathcal{P}(\Omega_0)$ on $[0,+\infty)$ such that 
\begin{equation}\label{Prop.Measurability_aux2}
P = P \otimes_{\tau_L}R
\end{equation}
up to $\tau_L$, where $P\otimes _{\tau_L} R$ is defined as in (\ref{def_gluedMeasure}).
We note that (\ref{Prop.Measurability_aux1}) yields $Q_x = \delta_{x}$ on $\mathcal{B}_{\tau(x)}^0$. If $\tau_L$ was a $(\mathcal{B}^0_t)_{t \geq 0}$-stopping time, this would imply 
\begin{equation}\label{PropMeasurability2_aux1}
Q_x\big(x' \in \Omega_0: \tau(x') = \tau(x)\big) = 1 
\end{equation}
for each $x \in \Omega_0$
and we could infer $P \otimes _{\tau_L} R = P$ on $\mathcal{B}^0_{\tau_L}$.
However, since $\tau_L$ is only a stopping time with respect to the right-continuous filtration $(\mathcal{B}_t)_{t \geq 0}$,  (\ref{PropMeasurability2_aux1}) does not follow from (\ref{Prop.Measurability_aux1}). Further, even with (\ref{PropMeasurability2_aux1}), we cannot infer $P = P \otimes_{\tau_L}R$ on $\mathcal{B}_{\tau_L}$, because it seems out of reach to show $Q_x = \delta_x$ on $\mathcal{B}_{\tau_L(x)}$. However, with (\ref{PropMeasurability2_aux1}) we do obtain (\ref{Prop.Measurability_aux2}) on $[0,\tau_L]$ in the sense that 
\begin{equation}\label{eq_P_and_ext_measure}
P(A) = P\otimes_{\tau_L}R(A), \quad A \in \sigma\big(\pi_{t \wedge \tau_L}, t \geq 0\big).
\end{equation} 
Indeed, choosing $A = \cap_{i \leq n}\{\pi_{t_i \wedge \tau_L} \in B_i\}$ with $0 \leq t_1 \leq \dots \leq t_{n}$, $B_i \in \mathcal{B}$ and $n \geq 1$, we have $$\bigcap_{i \leq n}\{\pi_{t_i \wedge \tau_L(x)} \in B_i\} \in \mathcal{B}^0_{\tau_L(x)}$$ and hence obtain, using \eqref{Prop.Measurability_aux1} and \eqref{PropMeasurability2_aux1},
$$Q_x(A) = Q_x\big(\cap_{i \leq n}\{\pi_{t_i \wedge \tau_L(x)} \in B_i\}\big) = \delta_x\big(\cap_{i \leq n}\{\pi_{t_i \wedge \tau_L(x)} \in B_i\}\big),$$
which yields (\ref{eq_P_and_ext_measure}). We also mention that the $\mathcal{B}_{\tau}$-measurability of $x \mapsto Q_x$ of Proposition \ref{PropMeasurability1} is only needed for the definition of $P \otimes _{\tau}R$ for which also mere $\mathcal{B}$-measurability would suffice. 
The above discussion leads to the following
\begin{proposition}\label{PropMeasurability2}
	For $\tau$ as in Proposition \ref{PropMeasurability1} and $x_0 \in L^2_{\sigma}$, let $P \in \mathcal{P}(\Omega_{0,\tau})$ be a martingale solution to \eqref{HSNSE_Intro} on $[0,\tau]$ in the sense of Definition \ref{DefMgSol_stop}. In addition to the situation in Proposition \ref{PropMeasurability1}, assume that there is $\mathcal{N} \in \Omega_{0,\tau}$ with $P(\mathcal{N}) = 0$ such that for every $x \in \mathcal{N}^c$ \eqref{PropMeasurability2_aux1} holds.
	Then the probability measure $P \otimes_{\tau}R \in \mathcal{P}(\Omega_0)$ defined by
	\begin{equation}\label{def_gluedMeasure}
	P\otimes_{\tau}R (\cdot) := \int_{\Omega_0}Q_x(\cdot)dP(x)
	\end{equation}
	is a martingale solution to \eqref{HSNSE_Intro} on $[0,+\infty)$ with initial condition $(0,x_0)$ and satisfies \eqref{eq_P_and_ext_measure}. 
\end{proposition}
\begin{proof}
	Due to the measurability of $x \mapsto Q_x$ mentioned in Proposition \ref{PropMeasurability1}, the measure $P \otimes_{\tau}R$ is well-defined on $\mathcal{B}$ and (\ref{eq_P_and_ext_measure}) follows as outlined in the paragraph above this proposition. Concerning the martingale property, properties (M1)-(M3) follow as in  \cite[Proposition 3.4.]{HZZ_stoNSE}. 
\end{proof}
Hence, once we have constructed a martingale solution $P$ on $[0,\tau_L]$ in Section 3 and 4, in order to extend $P$ to $[0,+\infty)$, it only remains to prove (\ref{PropMeasurability2_aux1}), which is addressed in Proposition \ref{PropMeasExtension} below. 

\subsection{Decomposition of \eqref{HSNSE_Intro} and treatment of linear stochastic part}
As in \cite{HZZ_stoNSE}, a key step towards Theorem \ref{Main_Result} is the construction of an analytically weak solution $u$ to \eqref{HSNSE_Intro} up to a suitable stopping time $\mathfrak{t}$ in a strong probabilistic sense, which violates an energy inequality on a sufficiently large set of paths on a prescribed time interval.\\
Fix a filtered probability space
$(\Omega, \mathcal{F}, (\mathcal{F}_t)_{t \geq 0}, \mathbb{\mathbf{P}}, B)$, where $B$ is a $GG^*$-Brownian motion and $(\mathcal{F}_t)_{t \geq 0}$ is the normal Brownian filtration, i.e. the canonical Brownian filtration augmented by all $\mathbb{\mathbf{P}}$-negligible sets.
\begin{theorem}\label{Theorem_analySol}
	Assume that for some $\sigma >0$ we have $\Tr \left[ A_{\alpha}^{\rho_{0}} G G^{*} \right] = \Tr \left[ (-\D)^{\rho_{0} \alpha} G G^{*} \right] < \infty$ for $\rho_{0} = \frac{5 + 2\sigma - 2 \alpha}{2 \alpha}$. Let $T >0$ and $K>1$ be given. Then there is $\gamma \in (0,1)$, a $\mathbb{\mathbf{P}}$-a.s. strictly positive $(\mathcal{F}_t)_{t \geq 0}$-stopping time $\mathfrak{t} = \mathfrak{t}(T,K)$ satisfying
	\begin{equation}\label{Theorem_analySol_largeStoppTime}
	\mathbf{P}(\mathfrak{t} \geq T) > \frac{1}{2} 
	\end{equation}
	and an $(\mathcal{F}_t)_{t \geq 0}$-adapted analytically weak solution $u$ to \eqref{HSNSE_Intro} with $u(\cdot \wedge \mathfrak{t}) \in C(\mathbb{R}_+,H^{\gamma})$ $\mathbf{P}$-a.s. such that $u(0) \in L^2_{\sigma}$ is deterministic. Furthermore, $u$ can be constructed such that
	\begin{equation}\label{u esssup finite}
	\underset{\omega \in \Omega}{\esssup}\,\underset{t \geq 0}{\sup}\,||u(t \wedge \mathfrak{t})||_{H^\gamma} < +\infty 
	\end{equation}
	and
	\begin{equation}\label{Failure_energy_ineq}
	||u(T)||_{L^2} > K\big(||u(0)||_{L^2}+T \cdot \Tr(GG^*)^{1/2}\big)
	\end{equation}holds pathwise on the set $\{\mathfrak{t} \geq T \}$.
\end{theorem}
In comparison with \cite[Thm 1.1.]{HZZ_stoNSE}, we remark that one may consider any $\kappa \in (0,1)$ instead of $\kappa = \frac{1}{2}$ in (\ref{Theorem_analySol_largeStoppTime}) by imposing an additional dependence for $\mathfrak{t}=\mathfrak{t}(T,K,\kappa)$. However, restricting to $\kappa = \frac{1}{2}$ is sufficient and slightly eases the subsequent presentation. 
\\

In order to construct $u$, we split \eqref{HSNSE_Intro} in two parts. One part is linear and contains the stochastic perturbation, while the second part is deterministic, but includes the nonlinear term. In this way, we separate the two challenging terms of \eqref{HSNSE_Intro} and deal with each of them by respective appropriate techniques, which fail to work directly on the full equation \eqref{HSNSE_Intro}. More precisely, following \cite{HZZ_stoNSE}, on $(\Omega, \mathcal{F}, (\mathcal{F}_t)_{t \geq 0}, \mathbf{P})$ we consider
\begin{equation}\tag{$H_{\text{lin,sto}}$}\label{H_lin_st}
\begin{cases}
dz +(-\D)^{\alpha}z + \nabla P_{1} dt &= dB_{t}, \\
\divv z &= 0, \\
z(0) &= 0.
\end{cases}
\end{equation}
and
\begin{equation}\tag{$H_{\text{nonlin,det}}$}\label{H_nonlin_det}
\begin{cases}
dv  +(-\D)^{\alpha}v + \text{div}\big((v+z)\otimes (v+z)\big) + \nabla P_{2} &=0, \\
\divv v &= 0,
\end{cases}
\end{equation}
where we first solve the stochastic equation (\ref{H_lin_st}) and fix its solution $z$ in (\ref{H_nonlin_det}). Concerning the former, by classical probabilistic arguments (cf. \cite[Theorem 5.4]{DPZ92}), we obtain a unique $(\mathcal{F}_t)_{t \geq 0}$-adapted solution $z: [0,+\infty) \to L^2_{\sigma}$ as a stochastic convolution. We stress that $z$ is a probabilistically strong solution. For later use, in the next proposition we establish the following crucial regularity of $z$. For the proof, consider Appendix \ref{AppA}.
\begin{proposition}\label{Prop_Reg_z_maintext}
Assume for some $\sigma >0$ we have $\Tr \left[ A_{\alpha}^{\rho_{0}} G G^{*} \right] = \Tr \left[ (-\D)^{\rho_{0} \alpha} G G^{*} \right] < \infty$ for $\rho_{0} = \frac{5 + 2\sigma - 2 \alpha}{2 \alpha}$. Then, for sufficiently small $\delta > 0$, the unique solution $z$ to \eqref{H_lin_st} satisfies
	 \begin{equation*}
	 \E_{\mathbf{P}} \left[ \| z \|_{C_{T}H^{\frac{5 + \sigma}{2}}} + \| z \|_{C^{\frac{1}{2} - 2\delta}_{T}H^{\frac{3 + \sigma}{2}}}  \right] < +\infty
	 \end{equation*}
for any $T >0$. 
\end{proposition}

\section{Proof of Theorem \ref{Theorem_analySol} by convex integration}\label{Section_CI}

In this section, we fix a probability space $(\Omega, \mathcal{F}, \mathbf{P})$ with a $GG^*$-Wiener process $B$ and the normal filtration $(\mathcal{F}_t)_{t \geq 0}$ generated by $B$. For fixed $\alpha \in (0,1/2)$, we consider the unique solution $z$ to (\ref{H_lin_st}) on this space and recall its regularity properties from Proposition \ref{Prop_Reg_z_maintext}. We intend to use convex integration techniques to solve the nonlinear equation (\ref{H_nonlin_det}) pathwise by an $(\mathcal{F}_t)_{t \geq 0}$-adapted analytically weak solution $v$ with increasing energy profile. To this end we need to control the size of $z = z(\omega)$ by which (\ref{H_nonlin_det}) deviates from the classical deterministic fractional Navier--Stokes equations considered in \cite{CDLDR18}. The following type of stopping time yields a suitable control. Let $\sigma>0$ as in Proposition \ref{Prop_Reg_z_maintext}. For $L>1$ and sufficiently small $\delta>0$, define $T_L$ on $\Omega$ as
\begin{equation}\label{def_stoppTime_T_L}
T_L(\omega) := \inf\{t \geq 0: ||z(t)||_{H^{\frac{5+\sigma}{2}}} \geq L^{\frac{1}{4}}C_S^{-1} \} \wedge \inf\{t \geq 0: ||z||_{C_t^{\frac{1}{2}-2\delta}H^{\frac{3+\sigma}{2}}} \geq L^{\frac{1}{2}}C_S^{-1} \} \wedge L.
\end{equation}
Here, $C_S$ comes from the Sobolev inequalities $||f||_{L^{\infty}} \leq C_S||f||_{H^\frac{3+\sigma}{2}}$, respectively. By standard Sobolev embeddings it follows that $z$ has a version in $C^0_{T_L,x}$.
By Proposition \ref{Prop_Reg_z_maintext} each $T_L$ is a bounded $(\mathcal{F}_t)_{t \geq 0}$-stopping time and we have $T_L >0$ and $T_L \nearrow +\infty$ as $L \to +\infty$ $\mathbf{P}$-a.s. Moreover, on $[0,T_L]$ we have the pathwise estimates
\begin{equation}\label{estim_z}
||z(t)||_{L^{\infty}} \leq L^{\frac{1}{4}},\quad ||\nabla z(t)||_{L^{\infty}} \leq L^{\frac{1}{4}}, \quad ||z||_{C_t^{\frac{1}{2}-2\delta}L^{\infty}} \leq L^{\frac{1}{2}}.
\end{equation}
Once we have constructed $v$ on $[0,T_L]$, we conclude that $u = z+v$ is an analytically weak solution to \eqref{HSNSE_Intro} on $[0,T_L]$ as stated in Theorem \ref{Thm3.1.-analog}, provided $L$ is chosen sufficiently large. Let us explain the iteration procedure leading to the construction of $v$. For each $q \in \mathbb{N}_0$, we construct a triple $(v_q,p_q,\mathring{R}_q)$, which analytically strongly solves the following \textit{fractional Navier--Stokes--Reynolds system} pathwise on $[0,T_L]$
\begin{equation}\label{Euler-Reynolds-system}
\begin{cases}
\partial_tv_q+\text{div}\big((v_q+z)\otimes (v_q + z)\big) + \nabla p_q +(-\Delta)^{\alpha}v_q &= \text{div}(\mathring{R}_q),\\
\text{div}(v_q) &= 0.
\end{cases}
\end{equation}
Here $v_q$ is a ($\omega$-wise) smooth vector field $v_q: [0,T_L]\times \mathbb{T}^3 \to \mathbb{R}^3$ with periodic boundary conditions and $\mathring{R}_q$ takes values in the space of symmetric, trace-free $3\times3$-matrices and continuously depends on $(t,x) \in [0,T_L]\times \mathbb{T}^3$ for fixed $\omega \in \Omega$. We will impose careful estimates on $v_q$ and $\mathring{R}_q$ and it will be crucial that both are $(\mathcal{F}_t)_{t \geq}$-adapted as function-valued processes. In contrast, the scalar pressure $p_q$ does not demand such considerations. Iteratively, we construct
$$v_{q+1} = v_{\ell} + w_{q+1},$$
where $v_{\ell}$ is a mollified version of $v_q$ as introduced below. The main characteristic of the \textit{perturbation} $w_{q+1}$ is the following: It consists of oscillatory waves with high frequency and a small amplitude. To capture the scale of both its frequency and amplitude we introduce sequences $(\lambda_q)_{q \geq 0}$ and $(\delta_q)_{q \geq 0}$, which diverge to $+\infty$ and converge to $0$, respectively. More precisely, for $a \in \mathbb{N}$, $c > \frac{5}{2}$ and $b >1$ we let
\begin{equation}\label{def_lambda_delta}
\lambda_q :=2a^{-b/2}a^{cb^{q+1}}, \quad \delta_q := \frac{a^b}{4(2\pi)^3}a^{-b^q}.
\end{equation}
In order to have $\lambda_q \in \mathbb{N}$, we assume $\frac{b}{2},c \in \mathbb{N}$. We remark that in order for $\lambda_q$ to be natural, one can drop the requirement $\frac{b}{2},c \in \mathbb{N}$ and instead assume $\lambda_q \in [2a^{cb^{q+1}-b/2},4a^{cb^{q+1}-b/2}] \cap \mathbb{N}$. However, since we do not aim to choose $b$ and $c$ close to their respective lower bounds $1$ and $\frac{5}{2}$, assuming $\frac{b}{2}$ and $c$ to be integers causes no harm in our situation. In view of Lemma \ref{Lemma_geom} we also assume $a$ to be a multiple of the geometric constant $n_0$ introduced in Appendix \ref{App_Beltrami_waves}. Note that $\delta^{1/2}_q\lambda_q$ is increasing in $q$. Furthermore, set $$M(t):= L^4e^{4Lt},\quad t \geq 0,$$
which can be thought of as a rapidly increasing energy profile for $v$, for large $L\in \mathbb{N}$ to be specified below.
The key estimates along the iteration are the following. For $t \in [0,T_L]$, and each $q \in \mathbb{N}_0$ we impose 
\begin{align}\
||v_{q+1}-v_{q}||_{C^0_{t,x}} &\leq M(t)^{\frac{1}{2}} \delta_{q+1}^{\frac{1}{2}}, \label{A.1}\tag{A.1}\\
||v_{q+1}-v_{q}||_{C_{t,x}^1} & \leq C_LM(t)^{\frac{1}{2}}\delta_{q+1}^{\frac{1}{2}}\lambda_{q+1}, \label{A.2}\tag{A.2}\\
||\mathring{R}_{q+1}||_{C^0_{t,x}} &\leq M(t)\delta_{q+2}c_R.	\label{A.3}\tag{A.3}
\end{align}
Here $c_R >0$ is a sufficiently small constant, whose size is only restricted by absolute constants via \eqref{c_R-restrictions} below. For the geometric number $C_{\mathbb{T}^3}:=||x||_{C^0} = \pi^{3/2}$, the constant 
\begin{equation}\label{def_C-L}
C_L := C_{\mathbb{T}^3}+\big(C_{\mathbb{T}^3}+1\big)\big(1+2M(L)^{1/2}+L^{1/4}\big)
\end{equation}
in (\ref{A.2}) is due to the first order term $||\Phi_j||_{C^1_{t,x}}$ of the transported vector field $\Phi_j$ introduced below. In contrast to the deterministic setting, where $||\Phi_j||_{C^1_{t,x}}$ is bounded by an absolute constant, which is then absorbed by $c_R$, (see \cite[Section 5]{BV19b}), here this is ruled out due to the dependence on the large parameter $L$, which cannot be absorbed by $c_R$ due to the requirement $c_RL \geq 20(2\pi)^3$ of Lemma \ref{Lem_Starting_triple}. In principle, one could roughly expect an exponential dependence of type $C_L^{2^q}$ not only in (\ref{A.2}), but also in (\ref{A.3}) and consequently also in (\ref{A.1}). However, the estimates for $\mathring{R}_{q+1}$ in Section \ref{testsection} allow to absorb all terms which explicitly contain $C_L$ and hence to obtain (\ref{A.3}) as stated above.

Before we state the main iteration result Proposition \ref{Prop_Main_Iteration}, we consider the initial stage $q = 0$.

\subsection{The starting triple}
For $L >1$ define
\begin{equation*}
v_0(t,x):= \frac{M(t)^{1/2}}{(2 \pi)^{3/2}}  \left( \cos( x_{3}), \sin( x_{3}), 0 \right),
\end{equation*}
\begin{equation*}
R_0(t,x):= \frac{(2L+1)  M(t)^{1/2}}{(2 \pi)^{3/2}}\begin{pmatrix} 0 & 0 & \sin( x_{3}) \\ 0 & 0 & -\cos( x_{3}) \\ \sin( x_{3}) & -\cos( x_{3}) & 0 \end{pmatrix}+ v_{0} \otimes z + z \otimes v_{0} + z \otimes z,
\end{equation*}
and $\mathring{R}_0$ as the traceless part of $R_0$. Furthermore, consider the scalar function $p_0$ such that
$$\text{div}(R_0) = \text{div}(\mathring{R}_0)-\nabla p_0,$$i.e.
\begin{equation*}
p_0(t,x) := -   \frac{1}{3}\big(2 v_0(t,x)\cdot z(t,x)+ |z(t,x)|^2\big).
\end{equation*} 
\begin{lemma}\label{Lem_Starting_triple}
 For $L>1$, the triple $(v_0,p_0,\mathring{R}_0)$ solves the Navier-Stokes-Reynolds equations \eqref{Euler-Reynolds-system} on $[0,T_L]\times \mathbb{T}^3$. Moreover, if we additionally assume
 \begin{equation}\label{cond_on_L}
20(2\pi)^3c_R^{-1} \leq L\leq \frac{a^2-2}{2},
 \end{equation}
 then the following estimates hold for $t \in [0,T_L]$.
 \begin{align}
 \label{Initial_stage_est_1}||v_0||_{C^0_{t,x}} &\leq M(t)^{1/2},\\
 \label{Initial_stage_est_2}||v_0||_{C^1_{t,x}} &\leq M(t)^{1/2}\delta_0^{1/2}\lambda_0,\\
 \label{Initial_stage_est_3}||\mathring{R}_0||_{C^0_{t,x}} &\leq M(t)c_R\delta_1.
 \end{align}
\end{lemma}

\begin{proof}
By a straightforward calculation it follows that $(v_0,p_0,\mathring{R}_0)$ solves (\ref{Euler-Reynolds-system}). (\ref{Initial_stage_est_1}) follows by definition of $v_0$. Concerning (\ref{Initial_stage_est_2}), we have
\begin{equation}
||v_0||_{C^1_{t,x}} \leq \frac{(2+2L)M(t)^{1/2}}{(2\pi)^{3/2}},
\end{equation}
which can be estimated by $$M(t)^{1/2}\delta_0^{1/2}\lambda_0 = \frac{M(t)^{1/2}}{(2\pi)^{3/2}}a^{cb-1/2},$$
provided $2+2L \leq a^{cb-1/2}$. Since $cb > 5/2$, this follows from (\ref{cond_on_L}).
For (\ref{Initial_stage_est_3}) we obtain
\begin{align*}
||R_0||_{C^0_{t,x}} \leq (2L+1)M(t)^{1/2}+M(t)^{1/2}L^{1/4}+L^{1/2} \leq 5\frac{M(t)}{L},
\end{align*}
which is bounded from above by $M(t)\delta_1c_R$ by (\ref{cond_on_L}), since $\delta_1^{-1} = 4(2\pi)^3$. 
\end{proof}

\subsection{Main iteration}
The following result captures the iteration procedure. 
\begin{proposition}[Main iteration]\label{Prop_Main_Iteration}
	Let $L \in \mathbb{N}$ satisfy \eqref{cond_on_L}. Then there is a choice of parameters $a > 0$, $b>1$ and $c > \frac{5}{2}$ such that there exists a sequence of triples $(v_q,p_q,\mathring{R}_q)_{q \geq 0}$ defined on $[0,T_L]\times \mathbb{T}^3$ with the following properties. $(v_0,p_0,\mathring{R}_0)$ is as in Lemma \ref{Lem_Starting_triple}, for each $q\geq 1$ the triple $(v_q,p_q,\mathring{R}_q)$ solves \eqref{Euler-Reynolds-system}, $v_q$ and $\mathring{R}_q$ are $(\mathcal{F}_t)_{t \geq 0}$-adapted and \eqref{A.1}-\eqref{A.3} hold. Furthermore, for each $q \geq 0$, $v_q(0)$ and $\mathring{R}_{q}(0)$ are deterministic
\end{proposition}
We first show how the construction of a solution $v$ to (\ref{H_nonlin_det}) and the proof of Theorem \ref{Theorem_analySol} follow from here.
\subsubsection*{Construction of $v$}
By Proposition \ref{Prop_Main_Iteration}, there is a sequence $(v_q,p_q,\mathring{R}_q)_{q \geq 0}$ which satisfies \eqref{A.1}-\eqref{A.3} subject to a sufficiently large choice of $a>0, b>1, c>\frac{5}{2}$. Interpolating the H\"older space $C^{\gamma}$ for $\gamma \in (0,\frac{1}{2cb})$ inbetween $C^0$ and $C^1$, we obtain for $t \in [0,T_L]$
\begin{equation*}
\sum_{q \geq 0}||v_{q+1}-v_q||_{C^{\gamma}_{t,x}} \leq \sum_{q \geq 0}||v_{q+1}-v_q||_{C^1_{t,x}}^{\gamma}\cdot ||v_{q+1}-v_q||_{C^0_{t,x}}^{1-\gamma} \leq C_L^{\gamma}M(t)^{\frac{1}{2}}\sum_{q \geq 0}\delta_{q+1}^{\frac{1}{2}}\lambda_{q+1}^{\gamma} \lesssim M(t)^{\frac{1}{2}}\sum_{q \geq 0}a^{b^{q+1}(cb\gamma-\frac{1}{2})} .
\end{equation*}
Hence, for $\gamma \in (0,\frac{1}{2bc})$, $(v_q)_{q \geq 0}$ converges pathwise in $C^{\gamma}_{T_L,x}$ to a unique limit $v$. Since \eqref{A.3} implies the convergence $\mathring{R}_q \underset{q \to \infty}{\longrightarrow}0$ uniformly in $(t,x) \in [0,T_L]\times \mathbb{T}^3$, by considering (\ref{Euler-Reynolds-system}) for $q \longrightarrow \infty$ it follows that $v$ is an analytically weak solution to (\ref{H_nonlin_det}) on $[0,T_L]$. As $C_L$ and $M(t)$ are independent of $\omega$, the above calculation in particular yields 
\begin{equation}\label{v finite esssup Hölder norm}
\underset{\omega \in \Omega}{\esssup}\, ||v||_{C^0_{T_L}C_x^{\gamma}} \leq C_LL^2e^{2L^2} < +\infty.
\end{equation}
Furthermore, due to the convergence in $C([0,T_L],C^{\gamma})$, the $(\mathcal{F}_t)_{t \geq 0}$-adaptedness of each $v_q$ implies the $(\mathcal{F}_t)_{t \geq 0}$-adaptedness of $v$ and, moreover, since each $v_q(0)$ is deterministic, so is $v(0)$.

Next, fixing $T>0$, we want to show that for sufficiently large $L \geq L_0(T)$, on $\{T_L \geq T\}$ we have
\begin{equation}\label{energy_ineq_to_show}
||v(T)||_{L^2} > \big(||v(0)||_{L^2}+L\big)e^{LT}.
\end{equation}
To this end, first of all we note that for sufficiently large $a>1$, for any $t \leq L$ we have on $\{T_L \geq t\}$
\begin{align}\label{ineq_initial_stage_1}
||v(t)-v_0(t)||_{L^2} \leq \sum_{q\geq 0} ||(v_{q+1}-v_q)(t)||_{L^2} \leq (2\pi)^{3/2}M(t)^{1/2}\sum_{q\geq 0} \delta_{q+1}^{1/2} \leq \frac{a^{b/2}}{2}M(t)^{1/2}\sum_{q \geq 0}a^{-b^{q+1}/2} \leq \frac{3}{4}M(t)^{1/2},
\end{align}
where we used (\ref{A.1}) for the second inequality and $b\geq 2$, which implies $b^q \geq qb$ for $q\geq 1$, for the last inequality. Consequently, we obtain
\begin{align}\label{aux_initial_stage_2}
||v(T)||_{L^2} \geq ||v_0(T)||_{L^2}-||v(T)-v_0(T)||_{L^2} = M(T)^{1/2} - ||v(T)-v_0(T)||_{L^2} \geq \frac{M(T)^{1/2}}{4}.
\end{align}
On the other hand, (\ref{ineq_initial_stage_1}) also yields
\begin{align*}
\big(||v(0)||_{L^2}+L\big)e^{LT} &\leq \bigg(||v_0(0)||_{L^2}+||v(0)-v_0(0)||_{L^2}+L\bigg)e^{LT}\\&
\leq \bigg(\frac{7}{4}M(0)^{1/2}+L\bigg)e^{LT}\\&
=\bigg(\frac{7}{4}L^2+L\bigg)e^{LT}\\&
< \frac{M(T)^{1/2}}{4},
\end{align*}
where the last inequality holds under the condition $7+4L^{-1} < e^{LT}$, which follows, if we impose
\begin{equation}\label{cond_L_large_T}
L > \frac{\ln(11)}{T}.
\end{equation}
Combining with (\ref{aux_initial_stage_2}), (\ref{energy_ineq_to_show}) follows.
\subsubsection*{Proof of Theorem \ref{Theorem_analySol}}
Let $T >0$ and $K>1$ from the theorem be given and let $L>1$, which we will choose sufficiently large in terms of $T$ and $K$ later on. Set
$$u := z+v,$$ where $z$ is the unique solution to (\ref{H_lin_st}) on the prescribed filtered probability space $(\Omega, \mathcal{F}, (\mathcal{F}_t)_{t \geq 0}, \mathbf{P}, B)$ and $v$ is the solution to (\ref{H_nonlin_det}) on $[0,T_L]$ constructed in the previous passage. By Proposition \ref{Regularity-z} and (\ref{v finite esssup Hölder norm}) it is clear that there is $\gamma >0$ such that $u \in C([0,T_L],H^{\gamma})$ $\mathbf{P}$-a.s., $u$ is $(\mathcal{F}_t)_{t \geq 0}$-adapted and $u(0) = v(0) \in L^2_{\sigma} $ is deterministic. By construction of $z$ and $v$ as solutions to the equations (\ref{H_lin_st}) and (\ref{H_nonlin_det}), respectively, it follows that $u$ is an analytically weak solution to \eqref{HSNSE_Intro} on $[0,T_L]$. Choosing $\mathfrak{t} = T_L$, the definition of $T_L$ and (\ref{v finite esssup Hölder norm}) imply that \eqref{u esssup finite} holds. Concerning (\ref{Failure_energy_ineq}), we can choose $L \geq L_0(K,T, ||u(0)||_{L^2}, G)$ such that in particular (\ref{cond_L_large_T}) holds and we have on $\{T_L \geq T\}$
\begin{align*}
||u(T)||_{L^2} \geq ||v(T)||_{L^2}-||z(T)||_{L^2} &> (||v(0)||_{L^2}+L)e^{LT}-(2\pi)^{\frac{3}{2}}L^{\frac{1}{4}} \\& = (||u(0)||_{L^2}+L)e^{LT}-(2\pi)^{\frac{3}{2}}L^{\frac{1}{4}} \\&>
K\big(||u(0)||_{L^2}+T \cdot \Tr(GG^*)^{1/2}\big),
\end{align*}
where we used \eqref{energy_ineq_to_show} and the definition of $T_L$ for the first strict inequality. Finally, upon possibly increasing $L$ once more to obtain $\mathbf{P}(T_L \geq T) > \frac{1}{2}$, the proof is complete. \qed

\subsection{Proof of the main iteration}
It remains to prove Proposition \ref{Prop_Main_Iteration}. To this end, let $q \geq 0$. Starting with the initial triple of Lemma \ref{Lem_Starting_triple}, assume triples up to $(v_q,p_q,\mathring{R}_q)$ fulfilling the iterative bounds \eqref{A.1}-\eqref{A.3} for a common choice of $a, b$ and $c$ are constructed already such that each $v_q$ and $\mathring{R}_q$ is $(\mathcal{F}_t)_{t \geq 0}$-adapted and $v_q(0)$ and $\mathring{R}_q(0)$ are deterministic. In all that follows, for functions $f$ depending on $(t,x,\omega) \in \mathbb{R}_+ \times \mathbb{R}^d \times \Omega$, we suppress the explicit dependence on $\omega$ and simply write $f(t,x)$. In order to construct $(v_{q+1},p_{q+1},\mathring{R}_{q+1})$ with all desired properties, we proceed as follows.

From the definition of $\lambda_q$ and $\delta_q$ in \eqref{def_lambda_delta}, the iterative construction $v_q = v_0+\sum_{p \leq q} (v_{p}-v_{p-1})$, the estimates (\ref{A.1}), (\ref{A.2}) and (\ref{Initial_stage_est_1})-(\ref{Initial_stage_est_3}), once we choose $a \geq a_0(b,c)$ sufficiently large, we obtain 
\begin{equation}
\sum_{p \geq 1}\delta_p^{1/2} \leq 1, \quad\quad \sum_{0 \leq p \leq q}\delta_p^{1/2}\lambda_p \leq 2\delta_q^{1/2}\lambda_q
\end{equation}and hence the bounds
\begin{equation}\label{estimate_v_q}
||v_q||_{C^0_{t,x}} \leq 2M(t)^{\frac{1}{2}}, \quad ||v_q||_{C^1_{t,x}}\leq 2C_LM(t)^{\frac{1}{2}}\delta_q^{\frac{1}{2}}\lambda_q, \quad t \in [0,T_L].
\end{equation}

\subsubsection{Choice of parameters}\label{Sect_parameter_choices}
One has to choose carefully the scales of the parameters $\ell = \ell_q$ and $\mu = \mu_q$ in comparison to the high frequency and low amplitude terms $\lambda_q$ and $\delta_q$. For $q \in \mathbb{N}_0$, set
\begin{equation}\label{def_ell_q}
\ell= \ell_q := \delta_{q+1}^{-1/8}\delta_q^{1/8}\lambda_q^{-1/4}\lambda_{q+1}^{-3/4}
\end{equation}
and
\begin{equation}\label{def_mu_q}
\mu = \mu_q := \delta_{q+1}^{1/4}\delta_{q}^{1/4}\lambda_q^{1/2}\lambda_{q+1}^{1/2}.
\end{equation}
This choice of $\mu$ and $\ell$ guarantees that
\begin{equation}\label{mu_ll_ellinv}
    \mu \ll \ell^{-1},
\end{equation}
which can be seen as follows: We have by definition and since $q \geq 0$ that
\begin{align*}
    \frac{1}{\mu \ell} = 2 (2 \pi)^{3/2} a^{-b/2} a^{b^{q} \left( \frac{1}{8} b + \frac{3}{8} - \frac{1}{4} bc + \frac{1}{4} b^{2} c \right)} \geq 2 (2 \pi)^{3/2} a^{-b/2} a^{\frac{1}{8} b + \frac{3}{8} - \frac{1}{4} bc + \frac{1}{4} b^{2} c} \gg 1,
\end{align*}
which holds since $b >1, c > \frac{5}{2}$. Moreover, for a fixed sufficiently small $\delta >0$ as in \ref{def_stoppTime_T_L}, we fix $b \in \mathbb{N}$ such that $b > \frac{8}{3}(\frac{1}{2}-2\delta)^{-1} > 5$, set 
\begin{align*}
\beta := \frac{b-1}{5b + 5} < \frac{1}{5},
\end{align*}
and fix $\varepsilon < \min \{ \frac{1}{4} -\frac{\alpha}{2}, \frac{\beta}{2} \}$. We also fix $N_0 := \lceil\frac{1+\varepsilon}{\beta}\rceil+1$ and 
\begin{align*}
c > \max \left( \frac{2}{\beta}, \frac{8}{3}\left(\frac{1}{2}-2\delta\right)^{-1}, \left(\frac{1}{2} - \alpha \right)^{-1} \right)  .
\end{align*} 
At this point, all parameters and length scales are fixed, with the exception of $a$, which throughout the iteration we will successively increase in terms of all other parameters as needed. With regard to Proposition \ref{Prop_Main_Iteration}, note that neither $a$ nor any other parameter depends on the stage $q$. These parameter choices yield the crucial estimates 
\begin{equation}\label{length_scales_original}
\frac{\delta_q^{1/2}\lambda_q\ell}{\delta_{q+1}^{1/2}}\ll 1, \quad\quad \frac{\delta_q^{1/2}\lambda_q}{\mu}+\frac{1}{\ell \lambda_{q+1}} \leq \lambda_{q+1}^{-\beta}\ll \frac{\delta_{q+2}}{\delta_{q+1}}, \quad\quad \lambda_{q+1}^{-1} \leq \frac{\delta_{q+1}^{1/2}}{\mu}.
\end{equation}
From here, we may further increase $a$ in terms of $L$ to obtain
\begin{equation}\label{length_scales_with_L-dep}
C_LM(L)^{1/2}\frac{\delta_q^{1/2}\lambda_q\ell}{\delta_{q+1}^{1/2}} \ll 1,\quad \frac{C_L}{\ell\lambda_{q+1}}+\frac{C_LM(L)^{\frac{1}{2}}\delta_{q}^{\frac{1}{2}}\lambda_q+L^{\frac{1}{4}}}{\mu} \leq \lambda_{q+1}^{-\beta} \ll \frac{\delta_{q+2} }{\delta_{q+1}}.
\end{equation}
Furthermore, upon choosing $a$ sufficiently large, the lower bounds for $b$ and $c$ stated above yield
\begin{equation}\label{length_scale_ell-with-delta}
C_L\ell^{1/2-2\delta}\delta_q^{1/2}\lambda_q \ll \delta_{q+2}.
\end{equation}
Moreover, the necessary relations for the small parameter $c_R$ are summarized by the condition
\begin{equation}\label{c_R-restrictions}
c_R < \text{min}\big(r_0^{2},(4Dc_0)^{-4}\big),
\end{equation}
where we denote by $c_0>0$ the maximum of all implicit constants appearing in Lemma \ref{Lemma_est_Phi-j} and \ref{Lemma_all_estimates}. We remark that all these constants are strictly positive and absolute, i.e. in particular they do not depend on any of the parameters involved in the iteration scheme mentioned above.

\subsubsection{Mollification}\label{Sect_Mollification}  In order to avoid a loss of derivative for $v_q$ and to improve the regularity of $\mathring{R}_q$ and $z$, we mollify in space and time. The time mollification needs to be non-anticipating in order to maintain $(\mathcal{F}_t)_{t \geq 0}$-adaptedness. Let $\{\phi_{\epsilon}\}_{\epsilon>0}$ be a family of standard mollifiers on $\mathbb{R}^3$ and $\{\varphi_{\epsilon}\}_{\epsilon >0}$ a family of standard mollifiers with support on $\mathbb{R}_+$. For technical reasons we replace $v_q$ and $z$ by $v_q(\cdot \wedge T_L)$ and $z(\cdot \wedge T_L)$, i.e. we consider their constant extensions beyond $T_L$ on $[0,L]$. However, we still denote these extended maps by $v_q$ and $z$ and remark that everything stated above in this section remains true for these extensions. For the mollification length scale $\ell = \ell_q$ defined in (\ref{def_ell_q}), set 
$$v_{\ell} := (v_q *_x \phi_{\ell})*_t \varphi_{\ell} , \quad\mathring{R}_{\ell} := (\mathring{R}_q *_x \phi_{\ell})*_t \varphi_{\ell},\quad z_{\ell} := (z *_x \phi_{\ell})*_t \varphi_{\ell}.$$
Note that $v_{\ell},\mathring{R}_{\ell}$ and $z_{\ell}$ are $(\mathcal{F}_t)_{t \geq 0}$-adapted, $z_{\ell}(0) = 0$, $v_{\ell}(0)$, $\partial_tv_{\ell}(0)$, $\mathring{R}_{\ell}(0)$ and $\partial_t\mathring{R}_{\ell}(0)$ are deterministic and $v_{\ell}$, $z_{\ell}$ and $\mathring{R}_{\ell}$ are divergence-free. It is straightforward to check that on $[0,T_L]$ the pair $(v_{\ell},\mathring{R}_{\ell})$ solves
\begin{equation}\label{Euler-Reynolds-system_MOLLIFIED}
\begin{cases}
\partial_tv_{\ell}+\text{div}\big((v_{\ell}+z_{\ell})\otimes (v_{\ell} + z_{\ell})\big) + \nabla p_{\ell} +(-\Delta)^{\alpha}v_{\ell} &= \text{div}(\mathring{R}_{\ell}+\mathring{R}_{\text{com}})\\
\text{div}(v_{\ell}) &= 0,
\end{cases}
\end{equation}
with
$$\mathring{R}_{\text{com}} := (v_{\ell}+z_{\ell})\mathring{\otimes} (v_{\ell}+z_{\ell}) - \big((v_q+z_q) \mathring{\otimes} (v_q+z_q)\big) *_x \phi_{\ell} *_t \varphi_{\ell}$$
and
$$p_{\ell}:= (p_q *_x \phi_{\ell})* _t \varphi_{\ell}- \frac{1}{3}\big(|v_{\ell}+z_{\ell}|^2-(|v_q+z_q|^2*_x \phi_{\ell})*_t \varphi_{\ell}\big).$$
From what we mentioned above, it follows that $\mathring{R}_{\text{com}}$ is $(\mathcal{F}_t)_{t \geq 0}$-adapted and $\mathring{R}_{\text{com}}(0)$ is deterministic.

By standard mollification estimates, the inductive estimates (\ref{A.3}) and (\ref{estimate_v_q}) and the definition of $T_L$ we obtain the following estimates for $N \geq 0$ and $t \in [0,T_L]$: 
\begin{align*}\label{Mollif_estim}
||v_q-v_{\ell}||_{C^0_{t,x}} &\leq \ell ||v_q||_{C_{t,x}^1} \leq 2C_LM(t)^{1/2}\delta_q^{1/2}\lambda_q\ell,\\
||v_{\ell}||_{C^{N+1}_{t,x}} &\lesssim \ell^{-N}||v_q||_{C^1_{t,x}}\lesssim C_LM(t)^{1/2}\delta_q^{1/2}\lambda_q\ell^{-N},\\
||v_{\ell}||_{C^0_{t,x}} &\leq ||v_q||_{C_{t,x}^0} \leq 2M(t)^{1/2},\\ \tag{M}
||\mathring{R}_{\ell}||_{C^N_{t,x}} &\lesssim \ell^{-N}||\mathring{R}_q||_{C^0_{t,x}} \lesssim \ell^{-N}M(t)\delta_{q+1}c_R,\\ 
 ||z_{\ell}||_{C^0_{t,x}} &\leq ||z||_{C^0_{t,x}} \leq L^{1/4},\\
 ||z_{\ell}||_{C^0_tC^{N+1}_x} &\lesssim \ell^{-N}L^{1/4},\\
 ||z-z_{\ell}||_{C^0_{t,x}} &\lesssim \ell^{1/2-\delta}L^{1/2}.
\end{align*}
All implicit constants stem from the mollifiers $\phi_{\ell}$ and $\varphi_{\ell}$ and hence only depend on $N$. 
Moreover, in order to estimate the solution $\Phi_j$ to the transport equation (\ref{Transport_eq}), we also need estimates on $v_{\ell}$ and $z_{\ell}$ beyond $T_L$. More precisely, for $t \in [0,L]$ and $N \geq 0$, we have
\begin{align*}
%\label{Mollif_est_extended_1}\textcolor{blue}{needed?}||v_q-v_{\ell}||_{C^0_{t,x}} &\leq \ell ||v_q||_{C^1_{T_L\wedge t,x}},\\
||v_{\ell}||_{C^0_{t,x}} &\leq ||v_q||_{C^0_{T_L \wedge t,x}}\leq 2M(T_L\wedge t)^{1/2},\quad\quad ||z_{\ell}||_{C^0_{t,x}} \leq ||z||_{C^0_{T_L \wedge t,x}} \leq L^{1/4},\\
\tag{M-ext}\label{Mollif_ext}||v_{\ell}||_{C^{N+1}_{t,x}} &\lesssim \ell^{-N}||v_q||_{C^1_{T_L\wedge t,x}} \lesssim C_LM(T_L\wedge t)^{1/2}\delta_{q}^{1/2}\lambda_q\ell^{-N},\\
||z_{\ell}||_{C^0_tC_x^{N+1}}&\lesssim \ell^{-N}L^{1/4}.
\end{align*}
\subsubsection{Time localization and phase transport}
In order to reduce the transport error $[\partial_t+(v_{\ell}+z_{\ell})\cdot \nabla]w_{q+1}$ of the perturbation, which is introduced in \ref{Def Rq+1}, for the principal part $w^{(p)}_{q+1}$ we consider the nonlinear phase $\Phi(t,x)$, where $\Phi$ is a vector field transported by $v_{\ell}+z_{\ell}$. In order to control its deviation from its initial value, we introduce a localization in time as follows. Let $\chi \in C^{\infty}_c\big((-\frac{3}{4},\frac{3}{4})\big)$ be a nonnegative cutoff function such that
$$
\sum_{l \in \mathbb{N}_0}\chi^2(t-l) = 1,\quad t \in \mathbb{R}.$$
Let $L \in \mathbb{N}$ be as in Proposition \ref{Prop_Main_Iteration}, $\mu = \mu_q \in \mathbb{N}$ as in (\ref{def_mu_q}) and set, for $j \in \{0,\dots,L\mu\}$, $\chi_j(t):= \chi(\mu t-j)$, which yields
\begin{equation}\label{chi_squared_part_unity}
\sum_{j}\chi^2_j(t) = 1, \quad t \in [0,L].
\end{equation}
Here and throughout, the summation in $j$ ranges over $\{0,\dots,L\mu\}.$ Since $\supp \chi_j  \subseteq \frac{1}{\mu}(-\frac{3}{4}+j,\frac{3}{4}+j)$, at each time $t$ at most two cutoffs are nontrivial. We recall that by $v_{\ell}$ and $z_{\ell}$ we always mean the mollification of $v_{q}(\cdot \wedge T_L)$ and $z(\cdot \wedge T_L)$, respectively. Consider $v_{\ell}+z_{\ell}$ as a smooth $2\pi$-periodic vector field on $[0,L]\times \mathbb{R}^3$. For $j \in \{0,\dots,L\mu\}$ we define $\Phi_j$: $[0,L]\times \mathbb{R}^3\times \Omega\to \mathbb{R}^3$ as the unique solution to the transport equation
\begin{equation}\label{Transport_eq}
\begin{cases}
[\partial_t + (v_{\ell}+z_{\ell})\cdot \nabla]\Phi_j &= 0,\\
\Phi_j(j\mu^{-1},x) &= x.
\end{cases}
\end{equation}
Note that $\Phi_j$ is the inverse flow of the ordinary differential equation with vector field $v_{\ell}+z_{\ell}$ with start at time $t = j\mu^{-1}$ as the identity. Thus, for each $t \in [0,L]$ and $x \in \mathbb{R}^3$, we have $\Phi_j(t,x)-\Phi_j(t,x+y) \in (2\pi\mathbb{Z})^3$ for any $y \in (2\pi\mathbb{Z})^3$. Consequently, $x \mapsto e^{i \lambda_{q+1}\xi\cdot \Phi_j(t,x)}$ is $(2\pi\mathbb{Z})^3$-periodic for each $t$ and may hence be considered an element in  $C^{\infty}(\mathbb{T}^3,\mathbb{C})$. Clearly, $\Phi_0(0)$ and $\partial_t\Phi_0(0)$ are deterministic and $(t,\omega) \mapsto \Phi_j(t,\omega) \in L^2$ is $(\mathcal{F}_t)_{t \geq 0}$-adapted.
To verify the inductive estimates (\ref{A.1})-(\ref{A.3}), we need the estimates on $\Phi_j$ contained in the following lemma, for which we recall the constant $C_L$ introduced in (\ref{def_C-L}).
\begin{lemma}\label{Lemma_est_Phi-j}
	For $j \in \{0,\dots,L\mu\}$, the unique solution $\Phi_j$ to (\ref{Transport_eq}) satisfies the following estimates. 
\begin{align}
\label{DPhi_C0_est}\| D \Phi_{j} \|_{C_{\supp \chi_j,x}^{0}} &\leq 1+C_{\mathbb{T}^3}, \\
\label{DPhi-Id_C0_est} ||D\Phi_j-\Id||_{C^0_{\supp \chi_j,x}}&\lesssim \frac{C_LM(L)^{1/2} \delta_{q}^{1/2} \lambda_{q} + L^{1/4}}{\mu} \ll 1, \\
\label{DPhi_CN_est} \| D \Phi_{j} \|_{C^0_{\supp \chi_j}C_x^{N}} &\lesssim \frac{C_LM(L)^{1/2}\delta_{q}^{1/2} \lambda_{q} + L^{1/4}}{\mu} \ell^{-{N}} \ll \ell^{-N},\quad N \geq 1,\\
\label{Phi_C^1tx_est}||\Phi_j||_{C^1_{\supp \chi_j,x}} &\leq C_L.
\end{align}
\end{lemma}
\begin{proof}
	\eqref{DPhi_C0_est} follows immediately from \eqref{DPhi-Id_C0_est}. The estimate \eqref{DPhi-Id_C0_est}, in turn, is a simple consequence of \cite[Proposition D.1, Eq. $(135)$]{BDLIS15}: 
	\begin{align*}
	&\| D \Phi_{j} - \Id \|_{C_{\supp \chi_j,x}^{0}} \\
	&\leq \frac{1}{\mu}\| D(v_{\ell}+ z_{\ell}) \|_{C_{\supp \chi_j,x}^{0}}  \exp\left(\frac{1}{\mu}   \| D(v_{\ell} + z_{\ell}) \|_{C_{\supp \chi_j,x}^{0}} \right) \\
	&\leq \frac{1}{\mu} \left( \| v_{q} \|_{C_{L,x}^{1}} +  L^{1/4} \right) \exp\left( \frac{1}{\mu} \left( \| v_{q} \|_{C_{L,x}^{1}} + L^{1/4} \right) \right) \\
	&\lesssim \frac{1}{\mu} \left(  C_LM(L)^{1/2} \delta_{q}^{1/2} \lambda_{q} + L^{1/4} \right) \exp \left(\frac{1}{\mu} \left(  C_LM(L)^{1/2} \delta_{q}^{1/2} \lambda_{q} + L^{1/4} \right) \right) \\
	&\lesssim \frac{C_LM(L)^{1/2} \delta_{q}^{1/2} \lambda_{q} + L^{1/4}}{\mu} \ll 1,
	\end{align*}
	where we used the extended mollification estimates (\ref{Mollif_ext}), and (\ref{length_scales_with_L-dep}) twice.
	Likewise, for \eqref{DPhi_CN_est}, we employ \cite[Proposition D.1, Eq. $(136)$]{BDLIS15} to obtain
	\begin{align*}
	\| D \Phi_{j} \|_{C^0_{\supp \chi_j}C_{x}^{N}} &\leq C_{N}  \ell^{-N} \frac{1}{\mu}( \| v_{q} \|_{C_{t,x}^{1}} +  L^{1/4} ) \exp\left(  \frac{C_{N}}{\mu} \left( 2 M(L)^{1/2} \delta_{q}^{1/2} \lambda_{q} + L^{1/4} \right) \right) \\
	&\lesssim \frac{C_LM(L)^{1/2}\delta_{q}^{1/2} \lambda_{q} + L^{1/4}}{\mu} \ell^{-{N}} \ll \ell^{-N}.
	\end{align*}
	Finally, (\ref{Phi_C^1tx_est}) follows via
	\begin{align*}
	||\Phi_j||_{C^1_{\supp \chi_j,x}} &= ||\Phi_j||_{C^0_{\supp \chi_j,x}}+||D\Phi_{j}||_{C^0_{\supp \chi_j,x}}+ ||\partial_t\Phi_{j}||_{C^0_{\supp \chi_j,x}}\\
	&\leq C_{\mathbb{T}^3}+(1+C_{\mathbb{T}^3})(1+||v_{\ell}+z_{\ell}||_{C^0_{L,x}})\\
	&\leq C_{\mathbb{T}^3}+(1+C_{\mathbb{T}^3})\big(1+2M(L)^{1/2}+L^{1/4}\big) = C_L,
	\end{align*}
	where we used \cite[Proposition D.1, Eq. $(132), (133)$]{BDLIS15}, (\ref{DPhi_C0_est}) and (\ref{Mollif_ext}).
\end{proof}

\subsubsection{Construction of \texorpdfstring{$w_{q+1}$}{wq1} and \texorpdfstring{$v_{q+1}$}{vq1}}
The velocity at stage $q+1$ will be defined as
\begin{equation}\label{def v_q+1}
v_{q+1} := v_{\ell}+w^{(p)}_{q+1}+w_{q+1}^{(c)},
\end{equation}
i.e. the perturbation $w_{q+1}$ consists of a \textit{principal term} $w^{(p)}_{q+1}$ and a \textit{corrector term} $w^{(c)}_{q+1}$. The former is constructed as a sum of highly oscillating Beltrami waves with low amplitude (see Appendix \ref{App_Beltrami_waves}), while the latter is needed to ensure $\text{div}(w_{q+1}) = 0$.

Let $\Lambda_0$, $\Lambda_1$ and $\gamma^{(0)}_{\xi}, \gamma^{(1)}_{\xi}$ for $\xi \in \Lambda_0, \Lambda_1$ be as in the geometric lemma \ref{Lemma_geom}. We set $\Lambda_j := \Lambda_0$ for $j \in 2\mathbb{N}_0$ and $\Lambda_j := \Lambda_1$ for $j \in 2\mathbb{N}_0+1$ and, likewise,  $\gamma^{(j)}_{\xi} := \gamma^{(0)}_{\xi}$ and $\gamma^{(j)}_{\xi} := \gamma^{(1)}_{\xi}$ for $j \in 2\mathbb{N}_0$ and $j \in 2\mathbb{N}_0+1$, respectively. On $[0,T_L]\times \mathbb{T}^3$, define
\begin{equation}
a_{j,\xi}(t,x):= a_{q+1,j,\xi}(t,x) := \chi_{j}(t) M(t)^{\frac{1}{2}}\delta_{q+1}^{\frac{1}{2}}c_R^{\frac{1}{4}}\gamma_{\xi}^{(j)}\bigg(\text{Id}-\frac{\mathring{R}_{\ell}(t,x)}{M(t)\delta_{q+1}c_R^{\frac{1}{2}}}\bigg)
\end{equation}
and set 
$$w^{(p)}_{q+1}(t,x):= \sum_{j}\sum_{\xi \in \Lambda_j}a_{j,\xi}(t,x)B_{\xi}e^{i\lambda_{q+1}\xi \cdot \Phi_j(t,x)}.$$
Note that in view of Lemma \ref{Lemma_geom}, we need 
$$\underset{t \in [0,T_L]}{\text{sup}}\frac{||\mathring{R}_{\ell}(t)||_{C^0}}{M(t)\delta_{q+1}c_R^{\frac{1}{2}}} < r_0,$$
which, considering \eqref{A.3}, holds due to \eqref{c_R-restrictions}. Since $\mathring{R}_{\ell}(0)$ and $\partial_t \mathring{R}_{\ell}(0)$ as well as $\Phi_0(0)$ and $\partial_t\Phi_0(0)$ are deterministic and since $\chi_j(0) = 0$ for $j \neq 0$, $w^{(p)}_{q+1}(0)$ and $\partial_tw^{(p)}_{q+1}(0)$ are deterministic as well. Moreover, the $(\mathcal{F}_t)_{t \geq 0}$-adaptedness of $\mathring{R}_{\ell}$ and each $\Phi_j$ yields $(\mathcal{F}_t)_{t \geq 0}$-adaptedness of $w^{(p)}_{q+1}$.

For future reference it is useful to introduce the notation
$$\phi_{j,\xi}(t,x):=\phi_{q+1,j,\xi}(t,x):=  e^{i\lambda_{q+1}\xi \cdot (\Phi_j(t,x)-x)}$$
and
$$W_{q+1,\xi}(x) := W_{\xi}(x) := B_{\xi}e^{i \lambda_{q+1}\xi \cdot x},$$
which we use to rewrite
$$w^{(p)}_{q+1}(t,x) = \sum_j \sum_{\xi \in \Lambda_j} a_{j,\xi}(t,x)\phi_{j,\xi}(t,x)W_{\xi}(x) = \sum_j \sum_{\xi \in \Lambda_j} a_{j,\xi}(t,x)W_{\xi}(\Phi_j(t,x)).$$
Moreover, we set $|\Lambda| := |\Lambda_{j}|$, which is independent of $j\in \mathbb{N}_0$ and, for $N_0 \in \mathbb{N}$ as in Section \ref{Sect_parameter_choices}, introduce the absolute constant
\begin{equation}\label{def_constant_D}
D:= 2|\Lambda|\underset{j,\xi}{\text{sup}}||\gamma_{\xi}^{(j)}||_{C^{N_0}(\overline{B_{r_0}(\text{Id})})}.
\end{equation}

The aforementioned cancellation of $w_{q+1}^{(p)}$ with the stress $\mathring{R}_{\ell}$ is captured by the following lemma, which is crucial for the estimate \eqref{A.3} for the new Reynolds stress $\mathring{R}_{q+1}$.
\begin{lemma}\label{Lem_cancellation_osc-error}
On $[0,T_L]\times \mathbb{T}^3$, we have 
\begin{equation*}
w_{q+1}^{(p)}\otimes w_{q+1}^{(p)} + \mathring{R}_{\ell} = M(t)\delta_{q+1}c_R^{\frac{1}{2}}\,\Id + \sum_{j,j', \xi+\xi' \neq 0}a_{j,\xi}a_{j',\xi'}\phi_{j,\xi}\phi_{j',\xi'}W_{\xi}\otimes W_{\xi'},
\end{equation*}
where the summation is understood to range over pairs $(j,\xi)$ and $(j',\xi')$ with $\xi \in \Lambda_j$ and $\xi' \in \Lambda_{j'}$.
\end{lemma}

\begin{proof}
	For abbreviation, denote the second summand of the right-hand side of the assertion by (II). By definition of $w_{q+1}^{(p)}$, since $\chi_j\chi_{j'} \equiv 0$ if $|j-j'| \geq 2$ and $\xi+\xi'=0$ for $\xi \in \Lambda_j$ and $\xi' \in \Lambda_{j'}$ implies $|j-j'| \in 2\mathbb{N}_0$, we have
	\begin{align*}
	w_{q+1}^{(p)}\otimes w_{q+1}^{(p)}(t,x) &= M(t)\delta_{q+1}c_R^{\frac{1}{2}}\sum_{j}\chi_j(t)^2\sum_{\xi \in \Lambda_j}\gamma_{\xi}^{(j)}\bigg(\text{Id}-\frac{\mathring{R}_{\ell}(t,x)}{M(t)\delta_{q+1}c_R^{\frac{1}{2}}}\bigg)^2B_{\xi}\otimes B_{-\xi}+ \text{(II)} \\ &=M(t)\delta_{q+1}c_R^{\frac{1}{2}}\sum_j\chi_j(t)^2\frac{1}{2}\sum_{\xi \in \Lambda_j}\gamma_{\xi}^{(j)}\bigg(\text{Id}-\frac{\mathring{R}_{\ell}(t,x)}{M(t)\delta_{q+1}c_R^{\frac{1}{2}}}\bigg)^2(\text{Id}-\xi \otimes \xi)+ \text{(II)}\\
	&=M(t)\delta_{q+1}c_R^{\frac{1}{2}}\sum_j\chi_j(t)^2\bigg(\text{Id}-\frac{\mathring{R}_{\ell}(t,x)}{M(t)\delta_{q+1}c_R^{\frac{1}{2}}}\bigg)+\text{(II)}\\
	&=M(t)\delta_{q+1}c_R^{\frac{1}{2}}\,\text{Id}-\mathring{R}_{\ell}+ \text{(II)}.
	\end{align*}
	Here we used \eqref{Beltrami_aux_1}, \eqref{eq_geom_lemma} and $\sum_j \chi_j^2 \equiv 1$ for the second, third and final equation, respectively. 
\end{proof}
Next, we introduce the corrector part $w^{(c)}_{q+1}$, which accounts for the fact that the principal part $w^{(p)}_{q+1}$ itself is not divergence-free. Setting
$$w^{(c)}_{q+1}(t,x) := \sum_{j}\sum_\xi\bigg[\frac{i}{\lambda_{q+1}}\nabla a_{j,\xi}(t,x)-a_{j,\xi}(t,x)\big(
D\Phi_j(t,x)-\text{Id}\big)\xi\bigg] \times W_{\xi}(\Phi_j(t,x)),\quad (t,x) \in [0,T_L]\times \mathbb{T}^3,$$
a direct calculation yields
$$w^{(p)}_{q+1}+w^{(c)}_{q+1} = \frac{1}{\lambda_{q+1}}\sum_j \sum_{\xi}\curl \big( i a_{j,\xi} \xi \times W_{\xi}(\Phi_j)\big),$$
which implies that the total perturbation
$$w_{q+1}:= w_{q+1}^{(p)}+w_{q+1}^{(c)}$$
is divergence-free. Since $w^{(c)}_{q+1}(0)$ and $\partial_tw^{(c)}_{q+1}(0)$ are deterministic and $w^{(c)}_{q+1}$ is $(\mathcal{F}_t)_{t \geq 0}$-adapted, together with the analogous observations for $w^{(p)}_{q+1}$ from above, it follows that $w_{q+1}(0)$ and $\partial_tw_{q+1}(0)$ are deterministic and $w_{q+1}$ is $(\mathcal{F}_t)_{t \geq 0}$-adapted. Finally, define $v_{q+1}$ as in (\ref{def v_q+1}) and note that $v_{q+1}(0)$ and $\partial_tv_{q+1}(0)$ are deterministic and that $v_{q+1}$ is $(\mathcal{F}_t)_{t \geq 0}$-adapted. Moreover, since by construction $w_{q+1}$ is smooth in $(t,x)$, so is $v_{q+1}$.

\subsubsection{Estimates for \texorpdfstring{$v_{q+1}-v_q$}{vq}}
Before we show (\ref{A.1}) and (\ref{A.2}), we collect useful estimates in Lemma \ref{Lemma_all_estimates}, which we use not only here, but also for the estimate (\ref{A.3}) on the new Reynolds stress later on. We define the coefficients of the full perturbation via 
\begin{equation*}
w_{q+1} = w_{q+1}^{(p)} + w_{q+1}^{(c)} =: \sum_{j,\xi} L_{j,\xi} e^{i \lambda_{q+1} \xi \cdot \Phi_{j}(t,x)} = \sum_{j,\xi} L_{j,\xi} \phi_{j,\xi} e^{i \lambda_{q+1} \xi \cdot x},
\end{equation*}
i.e.
\begin{align*}
L_{j,\xi} &:= \left( a_{j,\xi} B_{\xi} +  \left( \frac{i}{\lambda_{q+1}} \nabla a_{j,\xi} - a_{j,\xi} (D \Phi_{j} - \text{Id}) \xi   \right)  \times B_{\xi} \right).
\end{align*}

\begin{lemma}\label{Lemma_all_estimates}
	For each $N \geq 0$,  $(j,\xi)$ with $\xi \in \Lambda_j$ and $t \in [0,T_L]$, we have the following estimates for the coefficients of the perturbation $w_{q+1}$.
	\begin{align}
	\label{MIT_eq_aLCN_est} \| a_{j,\xi} \|_{C_{t}^{0} C_{x}^{N}} + \| L_{j,\xi} \|_{C_{t}^{0} C_{x}^{N}}&\lesssim M(t)^{1/2} \delta_{q+1}^{1/2} \ell^{-N}, \\
	\label{MIT_eq_phiCN_est}\| \phi_{j,\xi} \|_{C_{\supp \chi_{j}}^{0} C_{x}^{N}} &\lesssim  \lambda_{q+1}^{(1-\beta)N}, \\
	\label{MIT_eq_aL_timeCN_est} \| \partial_{t} a_{j,\xi}  \|_{C_{t}^{0} C_{x}^{N}} + \| \partial_{t} L_{j,\xi}  \|_{C_{t}^{0} C_{x}^{N}}  &\lesssim M(t)^{1/2} \delta_{q+1}^{1/2} \ell^{-(N+1)}, \\
	\label{MIT_eq_aL_materialCN_est}  \| (\partial_{t} + (v_{\ell} + z_{\ell}) \cdot \nabla ) a_{j,\xi} \|_{C_{t}^{0}C_{x}^{N}} + \| (\partial_{t} + (v_{\ell} + z_{\ell}) \cdot \nabla ) L_{j,\xi} \|_{C_{t}^{0} C_{x}^{N} } &\lesssim M(t) \delta_{q+1}^{1/2}  \ell^{-(N+1)}.
	\end{align}
\end{lemma}

\begin{proof}
	For \eqref{MIT_eq_aLCN_est}, by (\ref{Mollif_estim}) and the chain and product rule \cite[Eq. $(127)$, $(130)$]{BDLIS15}, we get
	\begin{align}\label{MIT_eq_aCN_est_INPROOF}
	\notag \| a_{j,\xi} \|_{C_{t}^{0} C_{x}^{N}} &\leq  c_{R}^{1/4} M(t)^{1/2} \delta_{q+1}^{1/2} \left\| \gamma_{\xi}^{(j)} \left( \text{Id} - \frac{\mathring{R}_{\ell} }{c_{R}^{1/2} M(t) \delta_{q+1}} \right) \right\|_{C_{t}^{0} C_{x}^{N}} \\\notag
	&\leq  c_R^{1/4}M(t)^{1/2} \delta_{q+1}^{1/2} \Bigg( \| \gamma_{\xi}^{(j)}  \|_{C^{1}}  \frac{ \left\|\mathring{R}_{\ell} \right\|_{C_{t}^{0} C_{x}^{N}} }{c_{R}^{1/2} M(t) \delta_{q+1}}  + \| \gamma_{\xi}^{(j)}  \|_{C^{N}} \Bigg( \frac{ \left\|\mathring{R}_{\ell} \right\|_{C_{t}^{0} C_{x}^{1}} }{c_{R}^{1/2} M(t) \delta_{q+1}} \Bigg)^{N} \Bigg) \\
	&\lesssim  c_R^{1/4}M(t)^{1/2} \delta_{q+1}^{1/2} \ell^{-N}. 
	\end{align}
	Similarly, we find by using \eqref{MIT_eq_aCN_est_INPROOF}, \eqref{length_scales_original} and Lemma \ref{Lemma_est_Phi-j} that
	\begin{align}\label{MIT_eq_LCN_est_INPROOF}
	\notag \| L_{j,\xi} \|_{C_{t}^{0} C_{x}^{N}} &\lesssim \| a_{j,\xi} \|_{C_{t}^{0} C_{x}^{N}} + \frac{1}{\lambda_{q+1}} \| \nabla a_{j,\xi} \|_{C_{t}^{0} C_{x}^{N}} +   \| a_{j,\xi} \|_{C_{t}^{0}C_{x}^{N}} \| D \Phi_{j} - \Id \|_{C_{\supp \chi_{j},x}^{0}} + \| a_{j,\xi} \|_{C_{t,x}^{0}} \| D \Phi_{j} - \Id \|_{C_{\supp \chi_{j}}^{0} C_{x}^{N}} \\
	&\notag \lesssim c_{R}^{1/4} M(t)^{1/2} \delta_{q+1}^{1/2} \ell^{-N} \left( 2 + \frac{1}{\lambda_{q+1} \ell} + \frac{2 M(L)^{1/2} C_{L} \delta_{q}^{1/2} \lambda_{q} + L^{1/4}}{\mu} \right)\\
	&\lesssim c_{R}^{1/4} M(t)^{1/2} \delta_{q+1}^{1/2} \ell^{-N}. 
	\end{align}
	The phase function is estimated by the chain rule, (\ref{Mollif_estim}), (\ref{length_scales_with_L-dep}) and \eqref{length_scales_original} as follows.
	\begin{align*}
	\| \phi_{j,\xi} \|_{C_{\supp \chi_{j}}^{0} C_{x}^{N}} &\lesssim \left( \lambda_{q+1} \| D \Phi_{j} \|_{C_{\supp \chi_{j}}^{0} C_{x}^{N-1}} + \lambda_{q+1}^{N} \| D \Phi_{j} - \text{Id} \|_{C_{\supp \chi_{j},x}^{0}}^{N} \right) \\
	&\lesssim  \left( \lambda_{q+1} \frac{\| v_{\ell} + z_{\ell} \|_{C_{t}^{0} C_{x}^{N}}}{\mu} e^{C_{N} \frac{1}{\mu} \| v_{\ell} + z_{\ell} \|_{C_{t}^{0} C_{x}^{1}} } + \lambda_{q+1}^{N} \left(  \frac{1}{\mu} \| v_{\ell} + z_{\ell} \|_{C_{t}^{0} C_{x}^{1}} e^{\frac{1}{\mu} \| v_{\ell} + z_{\ell} \|_{C_{t}^{0} C_{x}^{1}}} \right)^{N} \right) \\
	&\lesssim  \left( \lambda_{q+1} \frac{1}{\mu} \| v_{\ell} + z_{\ell} \|_{C_{t}^{0} C_{x}^{1}} e^{C_{N} \frac{1}{\mu} \| v_{\ell} + z_{\ell} \|_{C_{t}^{0} C_{x}^{1}} } \ell^{-N + 1} + \lambda_{q+1}^{N} \left(  \frac{1}{\mu} \| v_{\ell} + z_{\ell} \|_{C_{t}^{0} C_{x}^{1}} e^{\frac{1}{\mu} \| v_{\ell} + z_{\ell} \|_{C_{t}^{0} C_{x}^{1}}} \right)^{N} \right) \\
	&\lesssim  \left( \ell \lambda_{q+1} \lambda_{q+1}^{-\beta} \ell^{-N} + \lambda_{q+1}^{N(1-\beta)} \right) \lesssim \lambda_{q+1}^{N(1-\beta)}.
	\end{align*}
	Now we turn to the estimates containing temporal derivatives. For brevity we suppress in our notation the argument of $\gamma_{\xi}^{(j)}$. First, we apply the product and chain rule as well as \eqref{MIT_eq_aLCN_est} for $N=0$, \eqref{Mollif_estim} and \eqref{mu_ll_ellinv}, and choose $a$ sufficiently large to have $L \leq \ell^{-1}$ to find 
	\begin{align*}
	\left\| \partial_{t} a_{j,\xi} \right\|_{C_{t,x}^{0}} &= c_{R}^{1/4} M(t)^{1/2} \delta_{q+1}^{1/2} \left\|   2L \chi_{j} \gamma_{\xi}^{(j)} - \chi_{j} (D \gamma_{\xi}^{(j)}) \frac{\partial_{t} \mathring{R}_{\ell}}{c_{R}^{1/2} M(t) \delta_{q+1}} + 4L \chi_{j} \gamma_{\xi}^{(j)} \frac{\mathring{R}_{\ell}}{c_{R}^{1/2} M(t) \delta_{q+1}} + \chi_{j}' \gamma_{\xi}^{(j)}  \right\|_{C_{t,x}^{0}} \\
	&\lesssim M(t)^{1/2} \delta_{q+1}^{1/2} \left( \ell^{-1} + \mu \right) \lesssim M(t)^{1/2} \delta_{q+1}^{1/2} \ell^{-1}.
	\end{align*}
	In a similar way, we find for higher spatial derivatives
	\begin{align*}
	\left[ \partial_{t} a_{j,\xi} \right]_{C_{t}^{0} C_{x}^{N}} &= c_{R}^{1/4} M(t)^{1/2} \delta_{q+1}^{1/2} \left\| \chi_{j} D^{N} \left( 2L \gamma_{\xi}^{(j)} + (D \gamma_{\xi}^{(j)}) \frac{\partial_{t} \mathring{R}_{\ell}}{c_{R}^{1/2} M(t) \delta_{q+1}} - 4L \gamma_{\xi}^{(j)} \frac{\mathring{R}_{\ell}}{c_{R}^{1/2} M(t) \delta_{q+1}} \right) + \chi_{j}' D^{N} \gamma_{\xi}^{(j)} \right\|_{C_{t,x}^{0}}.
	\end{align*}
	Estimating as for the case $N=0$, the claimed inequality for $a_{j,\xi}$ follows.
	
	Concerning $L_{j,\xi}$, for $N \geq 0$, we obtain, using \eqref{Transport_eq} as well as product and chain rule, the previous estimates and (\ref{Mollif_estim})
	\begin{align*}
	\left\| \partial_{t} L_{j,\xi} \right\|_{C_{t}^{0} C_{x}^{N}} &\leq \left\| \partial_{t} a_{j,\xi} \right\|_{C_{t}^{0} C_{x}^{N}} + \frac{1}{\lambda_{q+1}} \left\| \nabla \partial_{t} a_{j,\xi} \right\|_{C_{t}^{0} C_{x}^{N}} + \| (\partial_{t} a_{j,\xi} ) (D \Phi_{j} - \text{Id}) + a_{j,\xi} D \partial_{t} \Phi_{j} \|_{C_{t}^{0} C_{x}^{N}} \\
	&\lesssim \left\| \partial_{t} a_{j,\xi} \right\|_{C_{t}^{0} C_{x}^{N}} + \frac{1}{\lambda_{q+1}} \left\| \nabla \partial_{t} a_{j,\xi} \right\|_{C_{t}^{0} C_{x}^{N}} +   \| \partial_{t} a_{j,\xi}  \|_{C_{t}^{0} C_{x}^{N}}  \| D \Phi_{j} - \Id \|_{C_{\supp \chi_{j},x}^{0}} + \| \partial_{t} a_{j,\xi}  \|_{C_{t,x}^{0}}  \| D \Phi_{j}  \|_{C_{\supp \chi_{j}}^{0} C_{x}^{N}}  \\
	&\quad +  \| a_{j,\xi} \|_{C_{t}^{0} C_{x}^{N}}  \| D \left[ (v_{\ell} + z_{\ell} ) \cdot \nabla \Phi_{j} \right] \|_{C_{\supp \chi_{j},x}^{0}} + \| a_{j,\xi} \|_{C_{t,x}^{0}}  \| D \left[ (v_{\ell} + z_{\ell} ) \cdot \nabla \Phi_{j} \right] \|_{C_{\supp \chi_{j}}^{0} C_{x}^{N}}  \\
	&\lesssim \left\| \partial_{t} a_{j,\xi} \right\|_{C_{t}^{0} C_{x}^{N}} + \frac{1}{\lambda_{q+1}} \left\| \partial_{t} a_{j,\xi} \right\|_{C_{t}^{0} C_{x}^{N+1}} +   \| \partial_{t} a_{j,\xi}  \|_{C_{t}^{0} C_{x}^{N}}  \| D \Phi_{j} - \Id \|_{C_{\supp \chi_{j},x}^{0}} + \| \partial_{t} a_{j,\xi}  \|_{C_{t,x}^{0}}  \| D \Phi_{j}  \|_{C_{\supp \chi_{j}}^{0} C_{x}^{N}} \\
	&\quad + \| a_{j,\xi} \|_{C_{t}^{0} C_{x}^{N}}  \left(  \| v_{\ell} + z_{\ell} \|_{C_{t}^{0} C_{x}^{1}} \| D \Phi_{j} \|_{C_{\supp \chi_{j},x}^{0}} + \| v_{\ell} + z_{\ell}  \|_{C_{t,x}^{0}} \| D \Phi_{j} \|_{C_{\supp \chi_{j}}^{0} C_{x}^{1}}  \right) \\
	&\quad + \| a_{j,\xi} \|_{C_{t,x}^{0}} \left(   \|  v_{\ell} + z_{\ell} \|_{C_{t}^{0} C_{x}^{N+1}} \| D \Phi_{j}  \|_{C_{\supp \chi_{j},x}^{0} } + \|  v_{\ell} + z_{\ell} \|_{C^0_{t,x}} \|  D \Phi_{j}  \|_{C_{\supp \chi_{j}}^{0} C_{x}^{N+1} } \right).
	\end{align*}
	We will show how to further estimate the terms in brackets of the penultimate line. The ones from the last line can be estimated in the same way. An application of Lemma \ref{Lemma_est_Phi-j}, \eqref{length_scales_with_L-dep} and \eqref{mu_ll_ellinv} yields
	\begin{align*}
        &\| v_{\ell} + z_{\ell} \|_{C_{t}^{0} C_{x}^{1}} \| D \Phi_{j} \|_{C_{\supp \chi_{j},x}^{0}} + \| v_{\ell} + z_{\ell}  \|_{C_{t,x}^{0}} \| D \Phi_{j} \|_{C_{\supp \chi_{j}}^{0} C_{x}^{1}} \\
        &=  \frac{\| v_{\ell} + z_{\ell} \|_{C_{t}^{0} C_{x}^{1}}}{\mu} \| D \Phi_{j} \|_{C_{\supp \chi_{j},x}^{0}} \mu + \| v_{\ell} + z_{\ell}  \|_{C_{t,x}^{0}} \| D \Phi_{j} \|_{C_{\supp \chi_{j}}^{0} C_{x}^{1}} \\
        &\lesssim \frac{ 2 M(L)^{1/2} \delta_{q}^{1/2} \lambda_{q} + L^{1/4}}{\mu} \mu + (2 M(L)^{1/2} + L^{1/4})\frac{C_LM(L)^{1/2}\delta_{q}^{1/2} \lambda_{q} + L^{1/4}}{\mu} \ell^{-{1}} \\
        &\leq \mu + (2 M(L)^{1/2} + L^{1/4}) \lambda_{q+1}^{-\beta} \ell^{-1} \\
        &\leq \mu + \ell^{-1} \lesssim \ell^{-1},
	\end{align*}
	for $a \geq a_{0}(L,\beta)$ sufficiently large to absorb the $L$-dependent constant in the penultimate line.

	Thus, altogether we find that
	\begin{align*}
        \left\| \partial_{t} L_{j,\xi} \right\|_{C_{t}^{0} C_{x}^{N}} &\lesssim M(t)^{1/2} \delta_{q+1}^{1/2} \ell^{-(N+1)} \left( 1 + \frac{1}{\lambda_{q+1} \ell} \right) \lesssim M(t)^{1/2} \delta_{q+1}^{1/2} \ell^{-(N+1)}.
	\end{align*}	
	The material derivatives can be treated by combining the previous estimates with the mollification estimates (\ref{Mollif_estim}):
	\begin{align}
	\nonumber \| (\partial_{t} + (v_{\ell} + z_{\ell}) \cdot \nabla ) a_{j,\xi} \|_{C_{t}^{0} C_{x}^{N}} &\lesssim  \| \partial_{t} a_{j,\xi} \|_{C_{t}^{0} C_{x}^{N}} +  \| v_{\ell} + z_{\ell} \|_{C_{t}^{0} C_{x}^{N}} \| \nabla a_{j,\xi} \|_{C_{t,x}^{0}} + \| v_{\ell} + z_{\ell} \|_{C_{t,x}^{0} } \| \nabla a_{j,\xi} \|_{C_{t}^{0} C_{x}^{N}} \\
	\nonumber &\lesssim  M(t)^{1/2} \delta_{q+1}^{1/2} \ell^{-(N+1)} \left( 1 + \| v_{q} \|_{C_{t,x}^{0}} + \| z \|_{C_{t,x}^{0}} \right) \\
	\nonumber &\lesssim M(t)^{1/2} \delta_{q+1}^{1/2} \ell^{-(N+1)}  \left( 1 + 2 M(t)^{1/2} + L^{1/4}  \right) \\
	\label{MIT_eq_a_timeCN_est_INPROOF} &\lesssim M(t) \delta_{q+1}^{1/2} \ell^{-(N+1)} ,
	\end{align}
	since $ L^{1/4} \leq M(t)^{1/2}$.  
	Similarly we find using \eqref{MIT_eq_a_timeCN_est_INPROOF}, \eqref{MIT_eq_aLCN_est} and (\ref{Mollif_estim})
	\begin{align*}
	\| (\partial_{t} + (v_{\ell} + z_{\ell}) \cdot \nabla ) L_{j,\xi} \|_{C_{t}^{0} C_{x}^{N}} &\lesssim \| \partial_{t} L_{j,\xi} \|_{C_{t}^{0} C_{x}^{N}} + \| v_{\ell} + z_{\ell} \|_{C_{t}^{0} C_{x}^{N}} \| \nabla L_{j,\xi} \|_{C_{t,x}^{0}} + \| v_{\ell} + z_{\ell} \|_{C_{t,x}^{0} } \| \nabla L_{j,\xi} \|_{C_{t}^{0} C_{x}^{N}}  \\
	&\lesssim M(t) \delta_{q+1}^{1/2} \ell^{-(N+1)},
	\end{align*}
	which completes the proof.
\end{proof}
From here the estimates (\ref{A.1}) and (\ref{A.2}) at stage $q+1$ are readily obtained as follows. (\ref{A.1}) follows from $v_{q+1}-v_q=w_{q+1} - (v_q-v_{\ell})$ and the estimates
\begin{align}\label{est_w-q+1_C0}
 ||w_{q+1}||_{C^0_{t,x}}\leq D||L_{j,\xi}||_{C^0_{t,x}} \lesssim Dc_R^{1/4}M(t)^{1/2}\delta_{q+1}^{1/2} \leq \frac{1}{2}M(t)^{1/2}\delta_{q+1}^{1/2}
\end{align}
and, using (\ref{Mollif_estim}),
\begin{align*}
||v_q-v_{\ell}||_{C^0_{t,x}} \lesssim C_LM(t)^{1/2}\delta_q^{1/2}\lambda_q\ell \ll M(t)^{1/2}\delta_{q+1}^{1/2}.
\end{align*}
Here we used (\ref{c_R-restrictions}) to absorb $D$ and the implicit absolute constants of (\ref{MIT_eq_aLCN_est}) and (\ref{length_scales_with_L-dep}). For future reference we also state the separate estimates 
\begin{align}\label{est_w-q+1-p_C0}
||w^{(p)}_{q+1}||_{C^0_{t,x}} \leq \frac{1}{4}M(t)^{1/2}\delta_{q+1}^{1/2}
\end{align}
and
\begin{align}\label{est_w-q+1-c_C0}
||w^{(c)}_{q+1}||_{C^0_{t,x}} \lesssim M(t)^{1/2}\delta^{1/2}_{q+1}\bigg(\frac{1}{\ell\lambda_{q+1}}+||D\Phi_j-\Id||_{C^0_{\supp \chi_{j},x}}\bigg),
\end{align}
for which we employed (\ref{MIT_eq_aCN_est_INPROOF}) and (\ref{c_R-restrictions}) once more.
Similarly, (\ref{A.2}) at stage $q+1$ follows from
\begin{align*}
||w_{q+1}||_{C^1_{t,x}} &\leq D\bigg(||L_{j,\xi}||_{C^1_{t,x}}+\big(\mu+\lambda_{q+1}||\Phi_j||_{C^1_{\supp \chi_{j},x}}\big)||L_{j,\xi}||_{C^0_{t,x}}\bigg)
 \\&
 \lesssim M(t)^{1/2}\delta_{q+1}^{1/2}\lambda_{q+1}\bigg(\frac{\ell^{-1}}{\lambda_{q+1}} + Dc_R^{1/4}\big(\frac{\mu}{\lambda_{q+1}}+C_L\big)\bigg)\\&
 \leq \frac 12 C_LM(t)^{1/2}\delta_{q+1}^{1/2}\lambda_{q+1}
\end{align*}
and
\begin{align*}
||v_q-v_{\ell}||_{C^1_{t,x}} \lesssim ||v_q||_{C^1_{t,x}} \lesssim C_LM(t)^{\frac{1}{2}}\delta_q^{\frac{1}{2}}\lambda_q \ll M(t)^{\frac{1}{2}}\delta_{q+1}^{\frac{1}{2}}\lambda_{q+1},
\end{align*}
where we used Lemma \ref{Lemma_est_Phi-j} and \ref{Lemma_all_estimates}, (\ref{MIT_eq_LCN_est_INPROOF}), (\ref{c_R-restrictions}), (\ref{Mollif_estim}),(\ref{estimate_v_q}) and we possibly further increased $a$ in terms of $L$ for the final inequality.

\subsubsection{Definition of \texorpdfstring{$\mathring{R}_{q+1}$}{Rq1}}\label{Def Rq+1} We continue with the definition of the new error term $\mathring{R}_{q+1}$. Since $(v_{\ell}, p_{\ell}, \mathring{R}_{\ell})$ solves (\ref{Euler-Reynolds-system_MOLLIFIED}) and the new triple $(v_{q+1},p_{q+1},\mathring{R}_{q+1})$ is supposed to solve (\ref{Euler-Reynolds-system}), we subtract (\ref{Euler-Reynolds-system_MOLLIFIED}) from (\ref{Euler-Reynolds-system}) to obtain
\begin{align}\label{calc_new_stress_and_p}
\notag \text{div}(\mathring{R}_{q+1})-\nabla p_{q+1} =&\, [\partial_t + (v_{\ell}+z_{\ell})\cdot \nabla]w_{q+1} \\&\notag 
 + \text{div}\big(w^{(p)}_{q+1}\otimes w^{(p)}_{q+1}+\mathring{R}_{\ell} \big) \\&\notag 
 + \text{div}\big(w^{(p)}_{q+1} \otimes w^{(c)}_{q+1} + w^{(c)}_{q+1} \otimes w_{q+1}\big) \\&
 + w_{q+1} \cdot \nabla(v_{\ell}+z_{\ell})\\&\notag 
 + (-\Delta)^{\alpha}w_{q+1} \\&\notag 
 + \text{div}\big(v_{q+1} \otimes (z-z_{\ell})+ (z-z_{\ell})\otimes v_{q+1} + z \otimes z - z_{\ell} \otimes z_{\ell}\big) \\&\notag 
 + \text{div}\big(R_{\text{com}}\big)-\nabla p_{\ell}.
\end{align}
We call the error terms on the right-hand side transport-, oscillation-, corrector-, Nash-, dissipation-, z- and commutator-error in their order of appearance.

 At this point, in order to define the new stress $\mathring{R}_{q+1}$, we need the following right-inverse to the $\text{div}$-operator, which we recall from \cite[Definition 4.2]{DLS13}. The operator $\mathcal{R}$ acts on vector fields $v \in C^{\infty}$ with $\int_{\mathbb{T}^3}vdx =0$ as
\begin{equation*}
(\mathcal{R}v)^{kl} = \big(\partial_k \Delta^{-1}v^{l}+\partial_l \Delta^{-1}v^k\big)-\frac{1}{2}\big(\delta_{kl}+\partial_k\partial_l\Delta^{-1}\big)\text{div}\Delta^{-1}v,
\end{equation*}
for $k,l \in \{1,2,3,\}$ and has the property that $\mathcal{R}v(x)$ is a symmetric, trace-free matrix for each $x \in \mathbb{T}^3$ and fulfills $\text{div}\big(\mathcal{R}v\big) = v$. If $v$ does not obey $\int_{\mathbb{T}^3}vdx = 0$, we write $\mathcal{R}v := \mathcal{R}\big(v-\int_{\mathbb{T}^3}vdx\big)$.

With $\mathcal{R}$ at hand, we consider the oscillation-error first. By Lemma \ref{Lem_cancellation_osc-error} and (\ref{Beltrami_aux_2}), we have
\begin{align*}
\text{div}\big(w^{(p)}_{q+1}\otimes w^{(p)}_{q+1}+\mathring{R}_{\ell} \big) &= \text{div}\big(\sum_{j,j', \xi+\xi' \neq 0}a_{j,\xi}a_{j',\xi'}\phi_{j,\xi}\phi_{j',\xi'}W_{\xi}\otimes W_{\xi'}\big) \\
&=\frac{1}{2}\sum_{j,j', \xi+\xi' \neq 0} a_{j,\xi}a_{j',\xi'}\phi_{j,\xi}\phi_{j',\xi'}\text{div}\big(W_{\xi}\otimes W_{\xi'}+W_{\xi'}\otimes W_{\xi}\big)\\&\quad+ \sum_{j,j', \xi+\xi' \neq 0}\big(W_{\xi}\otimes W_{\xi'}\big)\nabla \big(a_{j,\xi}a_{j',\xi'}\phi_{j,\xi}\phi_{j',\xi'}\big) \\
&=\frac{1}{2}\sum_{j,j', \xi+\xi' \neq 0} a_{j,\xi}a_{j',\xi'}\phi_{j,\xi}\phi_{j',\xi'}\nabla\big(W_{\xi}\cdot W_{\xi'}\big)\\&\quad+ \sum_{j,j', \xi+\xi' \neq 0}\big(W_{\xi}\otimes W_{\xi'}\big)\nabla \big(a_{j,\xi}a_{j',\xi'}\phi_{j,\xi}\phi_{j',\xi'}\big)\\
&= \text{div}(R_{\text{osc}})+\nabla p_{\text{osc}},
\end{align*}
where we set
$$R_{\text{osc}} :=\sum_{j,j', \xi+\xi' \neq 0} \mathcal{R}\bigg(\bigg(W_{\xi}\otimes W_{\xi'}-\frac{W_{\xi}\cdot W_{\xi'}}{2}\text{Id}\bigg)\nabla\big(a_{j,\xi}a_{j',\xi'}\phi_{j,\xi}\phi_{j',\xi'}\big)\bigg)$$
and
$$p_{\text{osc}} := \frac{1}{2}\sum_{j,j', \xi+\xi' \neq 0}a_{j,\xi}a_{j',\xi'}\phi_{j,\xi}\phi_{j',\xi'}\big(W_{\xi}\cdot W_{\xi'}\big).$$
Note that $R_{\text{osc}}(0)$ is deterministic and $R_{\text{osc}}$ is $(\mathcal{F}_t)_{t \geq 0}$-adapted and smooth.
Concerning the other error terms in (\ref{calc_new_stress_and_p}), we set
\begin{align*}
R_{\text{tra}} &:=  \mathcal{R}\big([\partial_t+(v_{\ell}+z_{\ell})\cdot \nabla]w_{q+1}\big),\\
R_{\text{corr}} &:= w^{(p)}_{q+1} \mathring{\otimes} w^{(c)}_{q+1} + w^{(c)}_{q+1} \mathring{\otimes} w_{q+1},\\
R_{\text{Nash}} &:= \mathcal{R}\big(w_{q+1} \cdot \nabla (v_{\ell}+z_{\ell})\big),\\
R_{\text{diss}} &:= \mathcal{R}\big((-\Delta)^{\alpha}w_{q+1}\big),\\
R_{\text{z}} &:= v_{q+1} \mathring{\otimes} (z-z_{\ell})+ (z-z_{\ell})\mathring{\otimes} v_{q+1} + z \mathring{\otimes} z - z_{\ell} \mathring{\otimes} z_{\ell}
\end{align*}
and $p_{\text{corr}} := \frac{1}{3\sum_{j,j', \xi+\xi' \neq 0}}\big(2\langle w_{q+1}^{(c)},w_{q+1}^{(p)}\rangle+|w^{(c)}_{q+1}|^2\big)$ and $p_{z}:=\frac{1}{3}\big(2\langle v_{q+1},z-z_{\ell}\rangle+|z|^2-|z_{\ell}|^2\big)$.
In view of (\ref{calc_new_stress_and_p}), now define 
$$\mathring{R}_{q+1}:= R_{\text{tra}}+R_{\text{osc}}+R_{\text{corr}}+R_{\text{Nash}}+R_{\text{diss}}+R_{\text{z}} + R_{\text{com}} $$
and
$$p_{q+1} := p_{\ell}-p_{\text{osc}}-p_{\text{corr}}-p_{z}.$$
Clearly, $\mathring{R}_{q+1}$ is trace-free. Moreover, inspecting each stress term defined above, it follows that $\mathring{R}_{q+1}(0)$ is deterministic and that $\mathring{R}_{q+1}$ is $(\mathcal{F}_t)_{t \geq 0}$-adapted. Since all terms in the definition of $\mathring{R}_{q+1}$ but $z$ are smooth and $z$ has a version in $C^0_{T_L,x}$, we have that $\mathring{R}_{q+1} \in C^0_{T_L,x}$. Moreover, by definition of $T_L$ we note that $\mathring{R}_{q+1}$ has weak first order spatial derivatives in $L^{\infty}_{T_L,x}$.

\subsubsection{Estimates for \texorpdfstring{$\mathring{R}_{q+1}$}{Rq1}}\label{testsection} We prove (\ref{A.3}) for $\mathring{R}_{q+1}$ at stage $q+1$ by considering the summands in the definition of $\mathring{R}_{q+1}$ separately. Let $t \in [0,T_L]$.
\subsubsection*{Estimates on \texorpdfstring{$R_{\mathrm{z}}$}{Rz}:} By the mollification estimates (\ref{Mollif_estim}), the definition of $T_L$, (\ref{estimate_v_q}) and (\ref{length_scale_ell-with-delta}), we obtain, choosing $a$ sufficiently large
\begin{align}\label{est_R-z}
||R_z||_{C^0_{t,x}} \leq \big(2||v_{q+1}||_{C^0_{t,x}}+||z||_{C^0_{t,x}}+||z_{\ell}||_{C^0_{t,x}}\big)\cdot ||z-z_{\ell}||_{C^0_{t,x}}
\lesssim M(t)\ell^{1/2-2\delta}
\ll M(t)\delta_{q+2}c_R. 
\end{align}

\subsubsection*{Estimates on \texorpdfstring{$R_{\mathrm{corr}}$}{Rcorr}:} By (\ref{est_w-q+1_C0}-\ref{est_w-q+1-c_C0}) and (\ref{DPhi-Id_C0_est}), we obtain
\begin{align}\label{est_R-corr}
\notag||R_{\text{corr}}||_{C^0_{t,x}} &\leq ||w_{q+1}^{(p)}||_{C^0_{t,x}}\cdot ||w_{q+1}^{(c)}||_{C^0_{t,x}}  +||w_{q+1}||_{C^0_{t,x}}\cdot ||w_{q+1}^{(c)}||_{C^0_{t,x}}\\&
\lesssim M(t)\delta_{q+1}\cdot \bigg(\frac{1}{\ell\lambda_{q+1}}+||D\Phi_j-\Id||_{C^0_{\supp \chi_{j},x}}\bigg)
 \ll M(t)\delta_{q+2}c_R,
\end{align}
where we have used (\ref{length_scales_with_L-dep}) and possibly increased $a$ in terms of $c_R$.

\subsubsection*{Estimates on \texorpdfstring{$R_{\mathrm{com}}$}{Rcom}:} By definition of $R_{\text{com}}$ and $T_L$, the mollification estimates (\ref{Mollif_estim}) and (\ref{estimate_v_q}), we have
\begin{align}\label{est_R-comm}
\notag||R_{\text{com}}||_{C^0_{t,x}} &\lesssim \ell ||v_q+z||_{C^0_{t,x}} \big(||v_q||_{C^1_{t,x}}+||z||_{L^{\infty}_tW^{1,\infty}_x}\big)+\ell^{1/2-2\delta}||v_q+z||_{C^0_{t,x}}\big(||v_q||_{C^1_{t,x}}+||z||_{C_t^{1/2-2\delta}L^{\infty}_x}\big) \\&\notag \lesssim
\ell^{1/2-2\delta}\big(M(t)^{1/2}+L^{1/4}\big)\big(C_LM(t)^{1/2}\delta_q^{1/2}\lambda_q+L^{1/2}\big)\\& \notag
\lesssim \ell^{1/2-2\delta}C_LM(t)\delta_q^{1/2}\lambda_q\\&
\ll M(t)\delta_{q+2}c_R,
\end{align}
where we used (\ref{length_scale_ell-with-delta}) for the final inequality.

The remaining error terms are estimated with the help of Proposition \ref{lemma_stationary_phase}. For the rest of this section we always choose $N = N_0$, where $N_0$ is as in Section \ref{Sect_parameter_choices}.

\subsubsection*{Estimates on \texorpdfstring{$R_{\mathrm{tra}} + R_{\mathrm{Nash}}$}{RtraRNash}:} Setting $D_t := \partial_{t} + (v_{\ell} + z_{\ell}) \cdot \nabla$, and using that the phase $\Phi_j$ is transported along $v_{\ell}+z_{\ell}$, we write 
\begin{align*} 
    (\partial_t+(v_{\ell}+z_{\ell})\cdot \nabla)w_{q+1} + w_{q+1} \cdot \nabla (v_{\ell}+z_{\ell}) = \sum_{j,\xi} \left[  D_t L_{j,\xi} + L_{j,\xi} \cdot \nabla (v_{\ell}+z_{\ell}) \right] \phi_{j,\xi} e^{i \lambda_{q+1}\xi \cdot x} =: \sum_{j,\xi} \Omega_{j,\xi} e^{i \lambda_{q+1}\xi \cdot x}.
\end{align*}
We then employ \eqref{MIT_eq_aLCN_est}, \eqref{MIT_eq_aL_materialCN_est}, (\ref{mu_ll_ellinv}) and (\ref{length_scales_with_L-dep}) to estimate
\begin{align*}
    \| \Omega_{j,\xi} \|_{C_{t,x}^{0}} &\leq \| D_t L_{j,\xi} \|_{C_{t,x}^{0}} + \| L_{j,\xi} \|_{C_{t,x}^{0}} \| v_{\ell}+z_{\ell} \|_{C_{t}^{0} C_{x}^{1}} \lesssim  M(t)^{1/2}  \delta_{q+1}^{1/2} \left( M(t)^{1/2} \ell^{-1} + \| v_{\ell} + z_{\ell} \|_{C_{t}^{0} C_{x}^{1} } \right) \\
    & \lesssim  M(t)  \delta_{q+1}^{1/2} \ell^{-1} \lesssim  M(t)  \delta_{q+1}^{1/2} \lambda_{q+1}^{1-\beta}.
\end{align*}
Similarly, we get, taking into account also derivatives of the phase function $\phi_{j,\xi}$ and using \eqref{MIT_eq_phiCN_est}:
\begin{align*}
    \| \Omega_{j,\xi} \|_{C_{t}^{0} C_{x}^{N}} &\lesssim  \| D_t L_{j,\xi} \|_{C_{t}^{0} C_{x}^{N}} + \| L_{j,\xi} \|_{C_{t}^{0} C_{x}^{N}} \| v_{\ell}+z_{\ell} \|_{C_{t}^{0} C_{x}^{1}} + \| L_{j,\xi} \|_{C_{t,x}^{0}} \| v_{\ell}+z_{\ell} \|_{C_{t}^{0} C_{x}^{N+1}} \\
    &\quad + \left( \| D_t L_{j,\xi} \|_{C_{t,x}^{0}} + \| L_{j,\xi} \|_{C_{t,x}^{0}} \| v_{\ell}+z_{\ell} \|_{C_{t}^{0} C_{x}^{1}} \right) \| \phi_{j,\xi} \|_{C_{\supp \chi_{j}}^{0} C_{x}^{N}} \\
    &\lesssim  M(t)^{1/2} \delta_{q+1}^{1/2} \ell^{-N} \left( M(t)^{1/2} \ell^{-1} + \| v_{\ell} + z_{\ell} \|_{C_{t}^{0} C_{x}^{1} } \right) +  M(t)  \delta_{q+1}^{1/2} \ell^{-1} \lambda_{q+1}^{N(1-\beta)} \\
    &\lesssim M(t)  \delta_{q+1}^{1/2} \left( \ell^{-(N+1)} + \ell^{-1} \lambda_{q+1}^{N(1-\beta)} \right) \lesssim  M(t)  \delta_{q+1}^{1/2} \lambda_{q+1}^{(N+1)(1-\beta)}.
\end{align*}
With these preparations, an application of the stationary phase lemma Proposition \ref{lemma_stationary_phase} yields:
\begin{align}\label{est_R-trash}
   \notag \| R_{\text{tra}} + R_{\text{Nash}} \|_{C_{t}^{0} C_{x}^{0}} &= \Big\| \sum_{j,\xi} \mathcal{R} \left( \Omega_{j,\xi} e^{i \lambda_{q+1}\xi \cdot x} \right) \Big\|_{C_{t,x}^{0}} \leq \Big\| \sum_{j,\xi} \mathcal{R} \left( \Omega_{j,\xi} e^{i \lambda_{q+1}\xi \cdot x} \right) \Big\|_{C_{t}^{0} C_{x}^{\varepsilon}} \\ \notag
    &\lesssim \sum_{j,\xi} \frac{\| \Omega_{j,\xi} \|_{C_{t,x}^{0}} }{\lambda_{q+1}^{1-\varepsilon}} + \frac{\| \Omega_{j,\xi} \|_{C_{t}^{0} C_{x}^{N_{0}}}}{\lambda_{q+1}^{N_{0} - \varepsilon}} + \frac{\| \Omega_{j,\xi} \|_{C_{t}^{0} C_{x}^{N_{0}+\varepsilon}}}{\lambda_{q+1}^{N_{0} }} \\ \notag
    &\lesssim  M(t)  \delta_{q+1}^{1/2} \left( \lambda_{q+1}^{\varepsilon - \beta} +  \lambda_{q+1}^{(N_{0}+1)(1-\beta) - N_{0} + \varepsilon} + \lambda_{q+1}^{ (N_{0}+1+\varepsilon)(1-\beta) - N_{0}}\right) \\
    &\leq M(t) \delta_{q+1}^{1/2} \left( \lambda_{q+1}^{\varepsilon - \beta} +  2 \lambda_{q+1}^{1 - \beta N_{0} - \beta + \varepsilon} \right) \ll M(t) c_{R} \delta_{q+2},
\end{align}
which translates via the definition of the parameters into the conditions on $b,c$:
\begin{align*}
    (\varepsilon - \beta) b^{2} c - \frac{1}{2}b + b^{2} = ((\varepsilon - \beta)c + 1)b^{2} - \frac{1}{2} b&<0, \\
    (1- \beta N_{0} + \varepsilon)b^{2} c - \beta b^{2} c - \frac{1}{2}b + b^{2} = b^{2} \left[ (1- \beta N_{0} + \varepsilon)c - \beta c + 1 \right] - \frac{1}{2}b &< 0,
\end{align*}
which are fulfilled by the choice of $\varepsilon$, $c$ and $N_0$ from Section \ref{Sect_parameter_choices}.

\subsubsection*{Estimates on \texorpdfstring{$R_{\mathrm{osc}}$}{Rosc}:} We set $f_{j,\xi,j',\xi'} := \nabla \big(a_{j,\xi}a_{j',\xi'}\phi_{j,\xi}\phi_{j',\xi'}\big) $. Then we have by \eqref{MIT_eq_aLCN_est}, \eqref{MIT_eq_phiCN_est} for $N \geq 0$
\begin{align*}
    \| f_{j,\xi,j',\xi'} \|_{C_{t}^{0} C_{x}^{N}} &\lesssim \| a_{j,\xi} \|_{C_{t}^{0} C_{x}^{N+1}} \| a_{j',\xi'} \|_{C_{t,x}^{0}} + \| a_{j,\xi} \|_{C_{t,x}^{0}} \| a_{j',\xi'} \|_{C_{t}^{0} C_{x}^{N+1}} \\
    &\quad + \| a_{j,\xi} \|_{C_{t,x}^{0}} \| a_{j',\xi'} \|_{C_{t,x}^{0}} \left( \| \phi_{j,\xi} \|_{C_{\supp \chi_{j}}^{0} C_{x}^{N+1}} + \| \phi_{j',\xi'} \|_{C_{\supp \chi_{j}}^{0} C_{x}^{N+1}}  \right) \\
    &\lesssim M(t) \delta_{q+1} \lambda_{q+1}^{(1-\beta)(N+1)}.
\end{align*}
The stationary phase lemma then yields
\begin{align}\label{est_R-osc}
    \notag\| R_{\text{osc}} \|_{C_{t,x}^{0}} &\lesssim \sum_{j,j', \xi+\xi' \neq 0} \frac{\| f_{j,\xi,j',\xi'} \|_{C_{t,x}^{0}} }{\lambda_{q+1}^{1-\varepsilon}} + \frac{\| f_{j,\xi,j',\xi'} \|_{C_{t}^{0} C_{x}^{N_{0}}}}{\lambda_{q+1}^{N_{0} - \varepsilon}} + \frac{\| f_{j,\xi,j',\xi'} \|_{C_{t}^{0} C_{x}^{N_{0}+\varepsilon}}}{\lambda_{q+1}^{N_{0}}} \\
    &\lesssim M(t) \delta_{q+1} \left( \lambda_{q+1}^{\varepsilon - \beta} + \lambda_{q+1}^{1-\beta N_{0} - \beta + \varepsilon} \right) \ll c_{R} M(t) \delta_{q+2},
\end{align}
by the same argument and assumptions as for the previous stress terms.

\subsubsection*{Estimates on \texorpdfstring{$R_{\mathrm{diss}}$}{Rdiss}:} Following the argument of \cite{CDLDR18}, we use the commutativity of $(-\D)^{\alpha}$ and $\mathcal{R}$, apply Schauder estimates (cf. \cite[Theorem 1.4]{RS16}, \cite[Theorem B.1]{CDLDR18}) and interpolation to estimate
\begin{align*}
    \| R_{\text{diss}} \|_{C_{t,x}^{0}} = \| (-\D)^{\alpha} \mathcal{R} w_{q+1} \|_{C_{t,x}^{0}} \lesssim C [\mathcal{R} w_{q+1} ]_{C_{t}^{0} C_{x}^{2\alpha + \varepsilon}} \lesssim \| \mathcal{R} w_{q+1} \|_{C_{t,x}^{0}}^{1-2 \alpha - \varepsilon} \| \mathcal{R} w_{q+1} \|_{C_{t}^{0} C_{x}^{1}}^{2 \alpha + \varepsilon}.
\end{align*}
We have by definition $w_{q+1} = \sum_{j,\xi} L_{j,\xi} \phi_{j,\xi} e^{i \lambda_{q+1} \xi \cdot x} =: \sum_{j,\xi} O_{j,\xi}  e^{i \lambda_{q+1} \xi \cdot x} $. By the product rule, \eqref{MIT_eq_aLCN_est} and \eqref{MIT_eq_phiCN_est}, we have
\begin{align*}
    \| O_{j,\xi} \|_{C_{t}^{0} C_{x}^{N}} \lesssim \| L_{j,\xi} \|_{C_{t}^{0} C_{x}^{N}} + \| L_{j,\xi} \|_{C_{t,x}^{0}} \| \phi_{j,\xi} \|_{C_{\supp \chi_{j}}^{0} C_{x}^{N}} \lesssim M(t)^{1/2} \delta_{q+1}^{1/2} \lambda_{q+1}^{(1-\beta)N}.
\end{align*}
By the stationary phase lemma, we find
\begin{align*}
    \| \mathcal{R} w_{q+1} \|_{C_{t,x}^{0}} &\lesssim \sum_{j,\xi} \frac{\| O_{j,\xi} \|_{C_{t,x}^{0}} }{\lambda_{q+1}^{1-\varepsilon}} + \frac{\| O_{j,\xi} \|_{C_{t}^{0} C_{x}^{N_{0}}}}{\lambda_{q+1}^{N_{0} - \varepsilon}} + \frac{\| O_{j,\xi} \|_{C_{t}^{0} C_{x}^{N_{0}+\varepsilon}}}{\lambda_{q+1}^{N_{0}}} \\
    &\lesssim M(t)^{1/2} \delta_{q+1}^{1/2} \left( \lambda_{q+1}^{\varepsilon - 1} + \lambda_{q+1}^{- \beta N_{0} + \varepsilon} \right) \lesssim M(t)^{1/2} \delta_{q+1}^{1/2} \lambda_{q+1}^{\varepsilon - 1}, %\ll c_{R} M(t)  \delta_{q+2},
\end{align*}
since $\varepsilon < 1$, for $a$ sufficiently large. Similarly, we can estimate 
\begin{align*}
    [ \mathcal{R} w_{q+1} ]_{C_{t}^{0} C_{x}^{1}} = \| \mathcal{R} D w_{q+1} \|_{C_{t,x}^{0}} &\lesssim \sum_{j,\xi} \frac{\| D O_{j,\xi} \|_{C_{t,x}^{0}} }{\lambda_{q+1}^{1-\varepsilon}} + \frac{\| D O_{j,\xi} \|_{C_{t}^{0} C_{x}^{N_{0}}}}{\lambda_{q+1}^{N_{0} - \varepsilon}} + \frac{\| D O_{j,\xi} \|_{C_{t}^{0} C_{x}^{N_{0}+\varepsilon}}}{\lambda_{q+1}^{N_{0}}} \\
    &\lesssim M(t)^{1/2} \delta_{q+1}^{1/2} \lambda_{q+1}^{\varepsilon} \left( \lambda_{q+1}^{- \beta} + \lambda_{q+1}^{1 - \beta N_{0} - \beta } \right) \lesssim M(t)^{1/2} \delta_{q+1}^{1/2} \lambda_{q+1}^{\varepsilon} %\ll c_{R} M(t)  \delta_{q+2},
\end{align*}
by the same arguments and assumptions as for $R_{\text{tra}} + R_{\text{Nash}}$.
Both estimates put together imply 
\begin{align}\label{est_R-diss}
    \| R_{\text{diss}} \|_{C_{t,x}^{0}} \lesssim \| \mathcal{R} w_{q+1} \|_{C_{t,x}^{0}}^{1-2 \alpha - \varepsilon} \| \mathcal{R} w_{q+1} \|_{C_{t}^{0} C_{x}^{1}}^{2 \alpha + \varepsilon} \lesssim M(t)^{1/2} \delta_{q+1}^{1/2} \lambda_{q+1}^{2 \alpha + 2\varepsilon - 1} \ll c_{R} M(t) \delta_{q+2},
\end{align}
if we have the condition 
\begin{align*}
    (2 \alpha + 2\varepsilon -1)c + 1 < 0,
\end{align*}
which holds by our choice of $\varepsilon < \frac{1}{2} - \alpha$ and $c > \frac{1}{1/2 - \alpha}$ from Section \ref{Sect_parameter_choices} and upon possibly increasing $a$ in terms of $\alpha, b$ and $c$.
\\

Finally, combining (\ref{est_R-z}), (\ref{est_R-corr}), (\ref{est_R-comm}), (\ref{est_R-trash}), (\ref{est_R-osc}) and (\ref{est_R-diss}), we obtain (\ref{A.3}) at stage $q+1$, which completes the verification of the inductive estimates (\ref{A.1})-(\ref{A.3}). Consequently, Proposition \ref{Prop_Main_Iteration} follows, which concludes the proof of Theorem \ref{Theorem_analySol}.

\section{Proof of the main result}
Having constructed the analytically weak solution $u$ of Theorem \ref{Theorem_analySol}, we proceed to the proof of our main result Theorem \ref{Main_Result}. Recall that $u = v+z$ is defined on a prescribed probability space $(\Omega,\mathcal{F},(\mathcal{F}_t)_t,\mathbb{\mathbf{P}},B)$, where $B$ is a $GG^*$-Wiener process and $(\mathcal{F}_t)_{t \geq 0}$ is the corresponding normal filtration. In particular, $(\mathcal{F}_t)_{t \geq 0}$ is right-continuous. For each $T$ and $K$ as in Theorem \ref{Theorem_analySol}, $u = u(T,K,L)$ is defined up to the $(\mathcal{F}_t)_{t \geq 0}$-stopping time $T_L$ defined in \eqref{def_stoppTime_T_L}, with $L>1$ sufficiently large as in the proof of Theorem \ref{Theorem_analySol} and with $\sigma >0$ as in Theorem \ref{Main_Result}. We recall that we have $T_L >0$ $\mathbb{\mathbf{P}}$-a.s. due to Proposition \ref{Prop_Reg_z_maintext}. Since $u \in C([0,T_L],H^{\gamma})$ for some $\gamma \in (0,1)$ as in Theorem \ref{Theorem_analySol}, we have that in particular $u_L:=u(\cdot \wedge T_L) \in \Omega_0$. Our goal is to prove that 
\begin{equation}\label{def_P}
P:= \mathbb{\mathbf{P}}\circ u_L^{-1}
\end{equation} 
is a martingale solution to \eqref{HSNSE_Intro} up to a suitable stopping time $\tau_L$. Let us construct $\tau_L$ first.

Similarly to \cite{HZZ_stoNSE}, on $\Omega_0 \cap L^{\infty}_{\mathrm{loc}}(\mathbb{R}_+,L^2_{\sigma})$ we consider the following processes with paths in $C(\mathbb{R}_+,H^{-3})$
$$M^{x}_{t,0}:= x(t)-x(0)+\int_0^tF_{\alpha}(x(r))dr,\quad x \in \Omega_0$$
and  
$$Z^x(t):= M^x_{t,0}+\int_0^t \mathbb{P}(-\Delta)^{\alpha}e^{(t-r)(-\Delta)^{\alpha}}M^x_{r,0}dr.$$
Then, set 
\begin{equation}\label{def_tauL}
\tau_L := \underset{n \to \infty}{\text{lim}}\tau^n_L,
\end{equation}
where the $\tau^n_L$ are nondecreasing in $n$ and defined as
$$\tau^n_L(x) := \text{inf}\bigg\{t \geq 0: ||Z^x(t)||_{H^{\frac{5+\sigma}{2}}} > (L-\frac{1}{n})^{\frac{1}{4}}C_S^{-1}\bigg\}\wedge \text{inf}\bigg\{t \geq 0: ||Z^x(t)||_{C_t^{\frac{1}{2}-{2\delta}}H^{\frac{3+\sigma}{2}}} > (L-\frac{1}{n})^{\frac{1}{2}}C_S^{-1}\bigg\}\wedge L,$$with $C_S$ as in \eqref{def_stoppTime_T_L}. Each $\tau^n_L$ is a $(\mathcal{B}_t)_{t \geq 0}$-stopping time, since $H_1=H^{\frac{5+\sigma}{2}}, H^{\frac{3+\sigma}{2}}$ and $H_2 = H^{-3}$ fulfill the assumptions of the following lemma, which is proved in  \cite[Lemma 3.5]{HZZ_stoNSE}.

\begin{lemma}\label{LemmaStoppTime}
	Let $(\Omega, \mathcal{F}, (\mathcal{F}_t)_{t \geq 0})$ be a filtered measurable space. Let $H_1,H_2$ be separable Hilbert spaces such that the embedding $H_1 \subseteq H_2$ is continuous. Suppose that there exists $\{h_k\}_{k \geq 1} \subseteq H_2^* \subseteq H_1^*$ such that for $f \in H_1$
	$$||f||_{H_1} = \underset{k \in \mathbb{N}}{\sup}\,h_k(f).$$
	Suppose that $X$ is an $(\mathcal{F}_t)_{t \geq 0}$-adapted process on $\Omega$ with trajectories in $C(\mathbb{R}_+,H_2)$. Then, for $L>1$ and $\alpha \in (0,1)$, both
	$$\tau_1 := \inf\{t \geq 0: ||X(t)||_{H_1} > L\} \quad\text{ and }\quad \tau_2 := \inf\{t \geq 0: ||X(t)||_{C_t^{\alpha}H_1} > L\}$$
	are $(\mathcal{F}_t^+)_{t \geq 0}$-stopping times, where $(\mathcal{F}_t^+)_{t \geq 0}$ denotes the right-continuous filtration of $(\mathcal{F}_t)_{t \geq 0}$.
\end{lemma}
Consequently $\tau_L$ is a bounded $(\mathcal{B}_t)_{t \geq 0}$-stopping time on $\Omega_0 \cap L^{\infty}_{\mathrm{loc}}(\mathbb{R}_+,L^2_{\sigma})$. While we might have $\tau_L(x) =0$ for irregular paths $x$, the regularity of $z$ and the fact that $u$ is an analytically weak solution to \eqref{HSNSE_Intro} allow to obtain $P(\tau_L >T) > \frac{1}{2}$ for any $T>0$ for sufficiently large $L$, see the following lemma. Note that $\mathbb{\mathbf{P}}\circ u_L^{-1}$ is in particular concentrated on $\Omega_0 \cap L^{\infty}_{\mathrm{loc}}(\mathbb{R}_+,L^2_{\sigma})$.
\begin{lemma}\label{Lem eq stopp times}
For $\mathbb{\mathbf{P}}$-a.e. $\omega \in \Omega$, we have
\begin{equation}\label{StoppingTimesEqual}
\tau_L(u_L(\omega)) = T_L(\omega).
\end{equation}
In particular, we have $\tau_L >0$ $P$-a.s. and for any $T>0$ there is $L_0 = L_0(T) >1$ such that for any $L \geq L_0$ it holds
\begin{equation}\label{eq_P_tau_L_>1/2}
P(\tau_L \geq T) > \frac{1}{2}.
\end{equation}
\end{lemma}
\begin{proof}
We start by proving $Z^{u_L} = z$ $\mathbb{\mathbf{P}}$-a.s., which is equivalent to 
\begin{equation}\label{Z_equal_z}
Z^u = z \quad \text{ on }[0,T_L]\,\, \mathbb{\mathbf{P}}-\text{a.s.}
\end{equation}
Since $u$ is an analytically weak solution to \eqref{HSNSE_Intro} on $[0,T_L]$, we have for each $e \in H^3$
$$\langle M^{u}_{t,0},e \rangle_{(-3,3)} = \langle u(t)-u(0),e\rangle_{L^2}+\int_0^t\langle F_{\alpha}(u(s)), e \rangle_{(-3,3)}ds = \langle B_t,e \rangle_{(-3,3)},\quad t \in [0,T_L]\,\, \mathbb{\mathbf{P}}-\text{a.s.}$$
Therefore we have $M^{u}_{t,0} = B_t$ for all $t \in [0,T_L]$ $\mathbf{P}$-a.s.
Hence, by definition of $Z$ and integration by parts, $Z^u$ solves \eqref{H_lin_st}, which implies \eqref{Z_equal_z}. From here, \eqref{StoppingTimesEqual} readily follows as in  \cite[Proposition 3.7.]{HZZ_stoNSE}.
The regularity \eqref{Reg_z} of $z$ and the definition of $T_L$ imply $T_L\nearrow +\infty$ $\mathbb{\mathbf{P}}$-a.s. as $L \rightarrow +\infty$. Hence, the $\sigma$-continuity of $\mathbf{P}$ yields \eqref{eq_P_tau_L_>1/2}. 
\end{proof}
We proceed by showing that the probability measure $P$ on $\Omega_{0,\tau}$ is a martingale solution to \eqref{HSNSE_Intro} on $[0,\tau_L]$. We point out that for the proof of Proposition \ref{Prop_P_is_locMGsol} it is essential that the analytically weak solution $u$ is probabilistically strong, i.e. it is $(\mathcal{F}_t)_{t \geq 0}$-adapted. 
\begin{proposition}\label{Prop_P_is_locMGsol}
	The probability measure $P \in \mathcal{P}(\Omega_{0,\tau_L})$ is a martingale solution to \eqref{HSNSE_Intro} on $[0,\tau_L]$ in the sense of Definition \ref{DefMgSol_stop} for some initial condition $(0,x_0)$, $x_0 \in L^2_{\sigma}$. 
\end{proposition}
\begin{proof}
	First, the construction of $u = v+z$ entails $u(0) = z(0)+v(0) = v(0)$. Since it is shown in Section \ref{Section_CI} that $v(0) \in L^2_{\sigma}$ is independent of $\omega \in \Omega$, \ref{M1} of Definition \ref{DefMgSol_stop} holds. We recall that the initial condition $x_0 = v(0)$ cannot be prescribed, but is an outcome of the construction of $v$ in Section \ref{Section_CI}.
	Secondly, by construction we have (with respect to $\mathbb{\mathbf{P}}$) 
	$$
	\underset{\omega \in \Omega}{\esssup}\, \underset{t \geq 0}{\text{sup}}\,||u(t \wedge T_L)||_{H^\gamma} < + \infty
	$$ 
	for some $\gamma >0$. In particular, for each $q \geq 1$ the left-hand side of \eqref{M3-ineq_stop} is bounded. Hence, there are functions $t \mapsto C_{t,q}$ as in \ref{M3} such that \eqref{M3-ineq_stop} holds. Finally, considering \ref{M2}, we have for $e \in H^3$, $t >r \geq 0$ and every continuous, bounded $\mathcal{B}_r$-measurable $g: \Omega_0 \to \mathbb{R}$
	\begin{align}\label{Prop4.2.aux1}
	\mathbb{E}_P\bigg[M^e_0(t\wedge \tau_L)g\bigg] = \mathbb{E}_{\mathbb{\mathbf{P}}}\bigg[M^e_0(t\wedge \tau_L(u_L),u_L)g(u_L)\bigg] = \mathbb{E}_{\mathbb{\mathbf{P}}}\bigg[M^e_0(r\wedge \tau_L(u_L),u_L)g(u_L)\bigg] = \mathbb{E}_P\bigg[M^e_0(r\wedge \tau_L)g\bigg].
	\end{align}
	Here, the second equality can be obtained as follows: Due to Lemma \ref{Lem eq stopp times}, we have $$\mathbb{E}_{\mathbb{\mathbf{P}}}\bigg[M^e_0(t\wedge \tau_L(u_L),u_L)g(u_L)\bigg] = \mathbb{E}_{\mathbb{\mathbf{P}}}\bigg[M^e_0(t\wedge T_L,u_L)g(u_L)\bigg].$$
	Since $u$ is an analytically weak solution to \eqref{HSNSE_Intro} on $[0,T_L]$, we have that $M^e_0(\cdot \wedge T_L,u_L):\mathbb{R}_+ \times \Omega \to \mathbb{R}$,
	\begin{align*}
	M^e_0(\cdot \wedge T_L,u_L) &= \langle u(\cdot\wedge T_L)-u(0),e \rangle_{(-3,3)}+\int_0^{\cdot\wedge T_L}\langle F_{\alpha}(u(r)),e \rangle_{(-3,3)}dr \\& = \langle B_{\cdot \wedge T_L},e \rangle_{L^2}
	\end{align*} is an $(\mathcal{F}_t)_{t \geq 0}$-martingale with respect to $\mathbb{\mathbf{P}}$. Hence, the second equality in (\ref{Prop4.2.aux1}) holds, since $g(u_L)$ is $\mathcal{F}_r$-measurable as a concatenation of $g$ with the $(\mathcal{F}_t)_{t \geq 0}$-adapted process $u_L$. 
\end{proof}
At this point, we could already compare $P$ to a classical Galerkin solution \cite{GRZ_MgEx09} on $[0,\tau_L]$ and, upon choosing $L$ sufficiently large, deduce non-uniqueness in law of martingale solutions on $[0,\tau_L]$. However, in view of Theorem \ref{Main_Result}, we first extend $P$ to $[0,+\infty)$  through the procedure of Section \ref{Section_measure_prelims}. To this end, note that Proposition \ref{PropMeasurability1} applies to $\tau_L$ and applying Proposition \ref{PropMeasurability2} to $P$ on $[0,\tau_L]$ yields the desired extended martingale solution $P\otimes_{\tau_L}R$, once we verify condition (\ref{PropMeasurability2_aux1}). This is achieved by the following result, which can be proven exactly as Proposition 3.8. in \cite{HZZ_stoNSE}.
\begin{proposition}\label{PropMeasExtension}
	For the martingale solution $P$ on $[0,\tau_L]$ and the stopping time $\tau_L$ constructed above, all conditions of Proposition \ref{PropMeasurability2} are fulfilled. Consequently $P \otimes_{\tau_L}R \in \mathcal{P}(\Omega_0)$ is a martingale solution to \eqref{HSNSE_Intro} on $[0,+\infty)$ with initial condition $(0,x_0) = (0,v(0))$ and \eqref{eq_P_and_ext_measure} holds.
\end{proposition}

Finally, we are prepared to conclude the proof of our main result Theorem \ref{Main_Result}.\\
\\
\textit{Proof of Theorem \ref{Main_Result}:} Let $T >0$ and $K =2$. By Theorem \ref{Theorem_analySol} and Propositions \ref{Prop_P_is_locMGsol}, \ref{PropMeasExtension}, there is $L>1$ and a martingale solution $P \otimes_{\tau_L}R$ to \eqref{HSNSE_Intro} on $[0,+\infty)$ with the following properties. $P$ is the law of $u_L = u (\cdot \wedge T_L)$ on $\Omega_0$ under $\mathbf{P}$, \eqref{eq_P_and_ext_measure} holds and we have $\mathbf{P}(T \geq T_L) > \frac{1}{2}$. Since $\mathbf{1}_{\{\tau_L \geq T\}}||x(T)||^2_{L^2} = \mathbf{1}_{\{\tau_L \geq T\}}||x(T\wedge \tau_L)||^2_{L^2}$ is $\sigma\big(\pi_{t \wedge \tau_L}, t\geq 0\big)$-measurable, (\ref{eq_P_and_ext_measure}) implies
\begin{align*}
\mathbb{E}_{P\otimes_{\tau_L}R}\bigg[\mathbf{1}_{\{\tau_L \geq T\}}||x(T)||^2_{L^2}\bigg]
= \mathbb{E}_P\bigg[\mathbf{1}_{\{\tau_L \geq T\}}||x(T)||^2_{L^2}\bigg].
\end{align*}
Therefore and by (\ref{StoppingTimesEqual}) and the failure of the energy inequality (\ref{Failure_energy_ineq}) with $K =2$, we obtain
\begin{align}\label{ProofMainResult_EnergyFailure}
\notag \mathbb{E}_{P\otimes_{\tau_L}R}\bigg[||x(T)||^2_{L^2}\bigg] &= \mathbb{E}_{P}\bigg[\mathbf{1}_{\{\tau_L \geq T\}}||x(T)||^2_{L^2}\bigg] +  \mathbb{E}_{P\otimes_{\tau_L}R}\bigg[\mathbf{1}_{\{\tau_L \leq T\}}||x(T)||^2_{L^2}\bigg] \\&  \geq \mathbb{E}_{\mathbb{\mathbf{P}}}\bigg[\mathbf{1}_{\{T_L \geq T\}}||u(T)||^2_{L^2}\bigg] > \mathbb{E}_{\mathbb{\mathbf{P}}}\bigg[\mathbf{1}_{\{T_L \geq T\}}4\big(||u(0)||^2_{L^2}+T\cdot \Tr(GG^*)\big)\bigg] \\& \notag >
2\big(||x_0||^2_{L^2}+T\cdot \Tr(GG^*)\big),
\end{align}
where $x_0 \in L^2_{\sigma}$ is the deterministic initial condition of $u$ from Theorem \ref{Theorem_analySol} and the martingale solution $P\otimes_{\tau_L}R$. \\
On the other hand, Theorem \ref{Thm3.1.-analog} and Remark \ref{Rem_simpleEnergyIneq_GalerkinSol} yield the existence of a second martingale solution $\tilde{P}$ to \eqref{HSNSE_Intro} on $[0,+\infty)$ with the same initial value $x_0$ such that 
$$\mathbb{E}_{\tilde{P}}\bigg[||x(T)||^2_{L^2}\bigg] \leq ||x_0||^2_{L^2}+T \cdot \Tr(GG^*).$$
Comparing with (\ref{ProofMainResult_EnergyFailure}), the martingale solutions $P\otimes_{\tau_L}R$ and $\tilde{P}$ on $[0,+\infty)$ are distinct on $[0,T]$, which concludes the proof.\qed 

\appendix
\section{A priori estimates for the linear equation}\label{AppA}
In this section, we provide the necessary \emph{a priori} estimates for the linear part  (\ref{H_lin_st}) of the stochastic hypodissipative equations, i.e. Proposition \ref{Prop_Reg_z_maintext}. For the reader's convenience, we give a full proof of the statement.

We consider the fractional Laplacian $A = A_{\alpha} = (-\D)^{\alpha}$ as an $L^{2}_{\sigma}$-based operator 
\begin{align*}
A \colon D(A) \subset L^{2}_{\sigma} \rightarrow L^{2}_{\sigma},
\end{align*}
with domain $D(A^{\alpha}) = {H}^{2 \alpha}$ for any $\alpha \in (0,1)$. To simplify the notation, we will suppress the subindex here and write $L^{2}$ instead of $L^{2}_{\sigma}$.

The following lemma collects important properties for the semigroup $S_{\alpha}(t)$ generated by the fractional Laplace operator $(-\D)^{\alpha}$. We include a simple proof for the convenience of the reader.

\begin{lemma}
	Let $(S_{\alpha}(t))_{t \geq 0}$ be the semigroup of linear operators in $L({L}^{2})$ generated by $A_{\alpha} = (-\D)^{\alpha}$. Then $(S_{\alpha}(t))_{t \geq 0}$ is an analytic, strongly continuous contraction semigroup. In particular, we have the estimates
	\begin{align}
	\nonumber \| S_{\alpha}(t) \|_{L({L}^{2})} &\leq 1, \quad t \geq 0, \\
	\label{AppA_eq_smoothing} \| A_{\alpha}^{\gamma} S_{\alpha}(t) \|_{L({L}^{2})} = \| (-\D)^{\alpha \gamma} S_{\alpha}(t) \|_{L({L}^{2})} &\leq C_{T,\gamma}  (t^{-\gamma} + 1) , \quad \forall \gamma > 0, t \in (0,T],
	\end{align}
\end{lemma}
\begin{proof}
	Since the operator $A_{\alpha}$ has the explicit Fourier series representation 
	\begin{align*}
	A_{\alpha} u(x) = \sum_{k \in {\Z}^{3}} |k|^{2 \alpha} \hat{u}_{k} e^{ik\cdot x},
	\end{align*}
	we can infer the corresponding Fourier series representation of the semigroup:
	\begin{align*}
	S_{\alpha}(t) u(x) := e^{t A}u(x) := \sum_{k \in {\Z}^{3}} e^{- |k|^{2 \alpha} t} \hat{u}_{k} e^{i k \cdot x}, \quad u \in {L}^{2}.
	\end{align*}
	That this is a strongly continuous semigroup can be easily checked. Furthermore, we have the following simple contraction bound, using Plancherel's theorem and estimating the exponential by 1
	\begin{align*}
	\left\| S_{\alpha}(t) u \right\|_{L^{2}}^{2} = \sum_{k \in {\Z}^{3}} e^{- 2|k|^{2 \alpha} t} | \hat{u}_{k}|^{2} \leq \sum_{k \in {\Z}^{3}} | \hat{u}_{k}|^{2} = \| u \|_{L^{2}}^{2} ~ \Longrightarrow ~ \left\| S_{\alpha}(t) \right\|_{L({L}^{2})} \leq 1. 
	\end{align*}
	We are left to prove the analyticity. To this end, let $t > 0$ and consider the $t$-derivative of $t \mapsto S_{\alpha}(t)$:
	\begin{align*}
	S'_{\alpha}(t) u(x) = \sum_{k \in {\Z}^{3}} - |k|^{2 \alpha} e^{- |k|^{2 \alpha} t} \hat{u}_{k} e^{i k \cdot x} = - \frac{1}{t} \sum_{k \in {\Z}^{3}}  t|k|^{2 \alpha} e^{- |k|^{2 \alpha} t} \hat{u}_{k} e^{i k \cdot x}.
	\end{align*}
	This implies, using again Plancherel's theorem and since the function $[0,\infty) \ni z \mapsto z^{2} e^{-2z}$ has global maximum $e^{-2}$, that
	\begin{align*}
	\| S'_{\alpha}(t) u \|_{L^{2}}^{2} = \frac{1}{t^{2}} \sum_{k \in {\Z}^{3}}  t^{2}|k|^{4 \alpha} e^{- 2|k|^{2 \alpha} t} | \hat{u}_{k} |^{2} \leq \frac{1}{t^{2}} e^{-2} \| u \|_{L^{2}}^{2} ~ \Longrightarrow ~ \| S'_{\alpha}(t) \|_{L({L}^{2})} \leq \frac{1}{t} e^{-2}.
	\end{align*}
	Therefore, by \cite[Proposition 2.1.9]{Lunardi95}, the assertion follows with $M_{0} = 1$, $M_{1} = e^{-2}$ and $\omega = 0$.
\end{proof}

The goal of this section is to prove the following proposition.
\begin{proposition}\label{Regularity-z}
	Assume that for some $\sigma >0$ we have $\Tr \left[ A_{\alpha}^{\rho_{0}} G G^{*} \right] = \Tr \left[ (-\D)^{\rho_{0} \alpha} G G^{*} \right] < \infty$ for $\rho_{0} = \frac{5 + 2\sigma - 2 \alpha}{2 \alpha}$, and let $T > 0$. 
	Then, for sufficiently small $\delta > 0$,
	\begin{equation}\label{Reg_z}
	\E \left[ \| z \|_{C_{T}H^{\frac{5 + \sigma}{2}}} + \| z \|_{C^{\frac{1}{2} - 2\delta}_{T}H^{\frac{3 + \sigma}{2}}}  \right] < \infty.
	\end{equation}
\end{proposition}

\begin{proof}
	The proof proceeds in a similar fashion to that of \cite[Proposition 34, p. 83]{Debussche13}. We use the \emph{factorization method} (cf. \cite[Section 5.3.1]{DPZ92}) to write
	\begin{equation*}
	z(t) = \int_{0}^{t} (t-s)^{\theta - 1}S_{\alpha}(t-s)Y(s) ds
	\end{equation*}
	for
	\begin{equation*}
	Y(s) = \frac{\sin (\pi \theta)}{\pi} \int_{0}^{s}  (s-r)^{-\theta} S_{\alpha}(s-r) G dW(r).
	\end{equation*}
	Define
	\begin{align*}
	j(\rho) = \begin{cases}
	\frac{\rho_{0}}{2}, \quad & \rho = \frac{5 + \sigma}{4\alpha} \\
	\rho, \quad &\rho = \frac{3 + \sigma}{4\alpha} \\
	0, \quad &\rho = 0.
	\end{cases}
	\end{align*}
	We first prove that for any $k \in \N$, $\psi = A_{\alpha}^{\rho} Y$ is in $L^{2k}(0,T;L^{2})$ $P$-a.s. for any of the three choices of $\rho$ above. Since $Y$ is Gaussian, we can estimate its higher moments by the second moment. Combining this with It\^{o}'s isometry and the estimate \eqref{AppA_eq_smoothing} we find
	\begin{align*}
	\E \left[ |A_{\alpha}^{\rho} Y(s) |_{L^{2}}^{2k} \right] &\leq c_{k} \left( \E \left[ |A_{\alpha}^{\rho} Y(s) |_{L^{2}}^{2} \right] \right)^{k} = c_{k} \left( \int_{0}^{s} (s-r)^{-2 \theta} |A_{\alpha}^{\rho} S_{\alpha}(s-r) G  |_{L_{2}}^{2} dr  \right)^{k} \\
	&\leq c_{k} \| A_{\alpha}^{j(\rho)} G \|_{L_{2}}^{2k} \left( \int_{0}^{s} (s-r)^{-2 \theta}  \| A_{\alpha}^{\rho - j(\rho) } S_{\alpha}(s-r) \|_{L(L^2)}^{2} d r \right)^{k} \\
	&\leq c_{k,\gamma} \| A_{\alpha}^{j(\rho)} G \|_{L_{2}}^{2k} \left( \int_{0}^{s} (s-r)^{-2 (\theta + \rho - j(\rho))} + (s-r)^{- 2 \theta}  d r \right)^{k}.
	\end{align*}
	The first factor is finite in all three cases since $\| A_{\alpha}^{\frac{\rho_{0}}{2}} G \|_{L_{2}}^{2} = \Tr \left[ A_{\alpha}^{\rho_{0}} G G^{*} \right] < \infty$ by assumption. For the integrals to be finite, we have the necessary condition 
	$$\rho - j(\rho) < \frac{1}{2}.$$ 
	Since the only $\rho$ that we want to consider that gives nonzero left-hand side is $\rho = \frac{5 + \sigma}{4 \alpha}$, this holds because $\rho_{0} = \frac{5 + 2\sigma - 2 \alpha}{2 \alpha} > \frac{5 + \sigma - 2 \alpha}{2 \alpha}$. 
	Therefore, to get finite integrals, we need to choose $\theta < \min \left\{ 1/2, 1/2 - \rho + j(\rho) \right\}$. This implies that $Y \in L^{2k}(\O \times [0,T]; D(A_{\alpha}^{\rho} ))$ for any $k \in \N$.
	
	Following \cite{DPZ96}, we define the (deterministic) convolution operator
	\begin{align*}
	R_{\theta,0} (\psi) := \int_{0}^{t} (t-s)^{\theta - 1} S_{\alpha}(t-s) \psi(s) ds, \quad \psi \in L^{2k}(0,T;L^{2}).
	\end{align*}
	We note that $z = R_{\theta,0}(Y)$ and $A_{\alpha}^{\rho} z = R_{\theta,0}(A_{\alpha}^{\rho} Y)$. By \cite[Proposition A.1.1 $(ii)$, p. 307]{DPZ96}, for any $\delta \in (0, \theta - \frac{1}{2k})$, $R_{\theta,0}$ is a bounded linear operator
	\begin{align*}
	R_{\theta,0} \colon L^{2k}(0,T; L^{2}) \rightarrow C^{\delta}([0,T];L^{2}) .
	\end{align*}
	Therefore, since $\frac{5 + \sigma}{4 \alpha} - \frac{\rho_{0}}{2} = \frac{1}{2} - \frac{\sigma}{4 \alpha}$, we can estimate for $0 < \delta < \min\{ \frac{1}{2}, \frac{\sigma}{4 \alpha} \} - \frac{1}{2k}$ and $k$ sufficiently large:
	\begin{align*}
	\E \left[ \| z \|_{C_{T} H^{\frac{5+\sigma}{2}}} \right] &\leq C_{\sigma} \E \left[ \| z \|_{C_{T} L^{2}} + \| A_{\alpha}^{\frac{5 + \sigma}{4 \alpha}} z \|_{C_{T} L^{2}} \right] \leq C_{\sigma} \E \left[ \| z \|_{C_{T}^{\delta} L^{2}} + \| A_{\alpha}^{\frac{5 + \sigma}{4 \alpha}} z \|_{C_{T}^{\delta} L^{2}} \right] \\
	&\leq C_{\sigma,k} \E \left[ \| Y \|_{L^{2k}(0,T; L^{2}) } \right] + C_{\sigma,k} \E \left[ \| A_{\alpha}^{\frac{5 + \sigma}{4\alpha}} Y \|_{L^{2k}(0,T; L^{2}) } \right] < \infty.
	\end{align*}
	And in a similar way we find for any $\delta >0 $ with $0 < \frac{1}{2} - 2\delta < \frac{1}{2} - \frac{1}{2k}$ and $k$ sufficiently large:
	\begin{align*}
	\E \left[ \| z \|_{C_{T}^{\frac{1}{2} - 2\delta} H^{\frac{3+\sigma}{2}}} \right] &\leq C_{\sigma} \E \left[ \| z \|_{C_{T}^{\frac{1}{2} - 2\delta} L^{2}} + \| A_{\alpha}^{\frac{3 + \sigma}{4 \alpha}} z \|_{C_{T}^{\frac{1}{2} - 2\delta} L^{2}} \right] \\
	&\leq C_{\sigma,k} \E \left[ \| Y \|_{L^{2k}(0,T; L^{2}) } \right]  + C_{\sigma,k}  \E \left[ \| A_{\alpha}^{\frac{3 + \sigma}{4\alpha}} Y \|_{L^{2k}(0,T; L^{2}) }  \right] < \infty. \qedhere
	\end{align*}
\end{proof}
\section{Proof of Theorem \ref{Thm3.1.-analog}}\label{AppB}
Here we prove both parts of Theorem \ref{Thm3.1.-analog}. Prior to the proof we state the following identities for the fractional Laplace operator $(-\Delta)^\alpha$ and the fractional Sobolev spaces $H^{\alpha}$, $\alpha \in (0,1)$, which readily follow via the definition of $(-\Delta)^\alpha$ as a Fourier multiplier and the density of the embedding $H^3 \hookrightarrow H^\alpha$. 
\begin{equation}\label{eq_H-alpha_norm_aux}
||y||^2_{H^{\alpha}} = \underset{z \in H^3, ||z||_{H^\alpha} \leq 1}{\sup}\langle y,z\rangle_{H^{\alpha}} = \underset{z \in H^3, ||z||_{H^\alpha} \leq 1}{\sup}\langle y, (1-\Delta)^\alpha)z\rangle_{L^2},\quad y \in H^\alpha,
\end{equation}
\begin{equation}\label{eq_frac_Laplace_pairing_aux1}
\langle (-\Delta)^{\alpha}y,z\rangle_{(-3,3)} = \langle y, (-\Delta)^{\alpha}z\rangle_{L^2},\quad y \in L^2_{\sigma}, z \in H^3
\end{equation}and
\begin{equation}\label{eq_frac_Laplace_pairing_aux2}
\langle (-\Delta)^\alpha y, y\rangle_{(-3,3)} = ||(-\Delta)^{\alpha/2}y||^2_{L^2}, \quad y \in H^3.
\end{equation}
We also need the following observation.
\begin{lemma}\label{Lem_ext_F-alpha}
	For $\alpha \in (0,1)$ the mapping $F_{\alpha}: y \mapsto -\divv(y\otimes y)-(-\Delta)^\alpha y$ extends from $H^1$ to an operator $F_{\alpha}: L^2_{\sigma} \to H^{-3}$ with
	$$\langle F_{\alpha}(y),z\rangle_{(-3,3)} = \langle y\otimes y, \nabla z \rangle_{L^2} - \langle y, (-\Delta)^{\alpha}z\rangle_{L^2}, \quad y \in L^2_{\sigma},\, z \in H^3.$$
\end{lemma}
\begin{proof}
	Considering (\ref{eq_frac_Laplace_pairing_aux1}), we only need to extend $y \mapsto \text{div}(y\otimes y)$ to $L^2_{\sigma}$. For $y \in H^1$, we have
	$$|\langle \text{div}(y\otimes y),z\rangle_{(-3,3)}| = |\langle \text{div}(y\otimes y),z\rangle_{L^2}| = |\langle y \cdot \nabla z, y\rangle_{L^2}| \leq ||\nabla z||_{L^{\infty}}||y||^2_{L^2} \leq C||y||^2_{L^2}, \quad ||z||_{H^3} \leq 1$$
	where $C>0$ is independent of $y$ and $z$ and comes from the Sobolev embedding $ H^2 \hookrightarrow L^{\infty}$. Now the claim follows by density of $H^1$ in $L^2_{\sigma}$.
\end{proof}
With these preparations, we proceed to the proof of Proposition \ref{Thm3.1.-analog}.
\\
\textit{Proof of \ref{Thm.3.1-analog_itm1}:} We aim to obtain the existence result by \cite[Thm. 4.6.]{GRZ_MgEx09}. Note that Definition \ref{DefMgSol} is compatible with \cite[Definition 3.1.]{GRZ_MgEx09} for the choices $Y=H=L^2_{\sigma}$, $X = H^{-3}$, the functions $$\mathcal{N}_q: Y \to [0,+\infty], \quad \mathcal{N}_q(y) = \begin{cases}
||y||_{L^2}^{2(q-1)}\cdot||y||^2_{H^\alpha},\, &\text{if } y \in H^{\alpha}, \\
+\infty,\, &\text{else}
\end{cases}$$
and drift and diffusion coefficient
$$A: Y \to X, \quad A(y) = F_{\alpha}(y)$$
and $B: Y \to L_2(U,H)$, $B \equiv G$, respectively.
Hence, we show that these choices satisfy all relevant assumptions of \cite{GRZ_MgEx09}, in particular the main conditions (C1)-(C3) of \cite[p.1733]{GRZ_MgEx09}.

First, the embeddings $H^3 \hookrightarrow L^2_{\sigma} \hookrightarrow L^2_{\sigma} \hookrightarrow H^{-3}$ are continuous and dense, and the first and the last embedding are also compact. In particular, there is an orthonormal basis of $L^2_{\sigma}$ in $H^3$. We choose the linear span of such basis as the countable set of test vector fields $\mathcal{E}$ in \cite{GRZ_MgEx09}.

Secondly, we note that all assumptions on the diffusion coefficient of \cite{GRZ_MgEx09} are fulfilled for our choice of the constant operator-valued map $B \equiv G \in L_2(U,L^2_{\sigma})$.

Next, $\mathcal{N}_1$ belongs to the class $\mathcal{U}^2$ of \cite{GRZ_MgEx09} and each $\mathcal{N}_q$ is lower semicontinuous on $L^2_{\sigma}$. Indeed, it is clear that $\mathcal{N}_1(0) = 0$ and $\mathcal{N}_1(cy) \leq c^2\mathcal{N}_1(y)$ for each $y \in L^2_{\sigma}$ and $c \geq 0$. Furthermore, $\{y \in L^2_{\sigma}: \mathcal{N}_1(y) \leq 1\}$ is relatively compact in $L^2_{\sigma}$, since the embedding $H^{\alpha}\hookrightarrow L^2_{\sigma}$ is compact. Concerning the lower semicontinuity of $\mathcal{N}_1$ on $L^2_{\sigma}$, for $y_n \underset{n \to \infty}{\longrightarrow}y$ in $L^2_{\sigma}$, we need to show
\begin{equation}\label{lsc_N1_aux1}
||y||_{H^\alpha}^2 \leq \underset{n \to \infty}{\liminf}\,||y_n||^2_{H^\alpha}.
\end{equation}
Hence, without loss of generality, assume $\underset{n \geq 1}{\text{sup}}\,||y_n||_{H^\alpha}^2 < +\infty$, i.e. $(y_n)_{n \geq 1}$ is bounded in $H^{\alpha}$ and for some subsequence and some $y' \in H^\alpha$ we have $y_{n_k} \underset{k \to \infty}{\longrightarrow}y'$ weakly in $H^{\alpha}$. Consequently $y = y'$, i.e. $||y||^2_{H^\alpha} < +\infty$. Now \eqref{lsc_N1_aux1} follows from \eqref{eq_H-alpha_norm_aux} via
\begin{align*}
||y||^2_{H^\alpha} &= \underset{\substack{z \in H^3,\\ ||z||_{H^\alpha} \leq 1}}{\text{sup}}|\langle y,(1-\Delta)^\alpha z\rangle_{L^2}| = \underset{\substack{z \in H^3,\\ ||z||_{H^\alpha} \leq 1}}{\text{sup}}\underset{n \to \infty}{\text{lim}}|\langle y_n,(1-\Delta)^\alpha z\rangle_{L^2}| \\
&\leq \underset{n \to \infty}{\liminf}\underset{\substack{z \in H^3,\\ ||z||_{H^\alpha} \leq 1}}{\text{sup}}|\langle y_n,(1-\Delta)^\alpha z\rangle_{L^2}| = \underset{n \to \infty}{\liminf}\,||y_n||^2_{H^\alpha}.
\end{align*}
It follows that also each $\mathcal{N}_q$ as defined above is lower semicontinuous as the product of a continuous and a lower semicontinuous nonnegative function.

To conclude the proof of \ref{Thm.3.1-analog_itm1}, it remains to verify conditions (C1)-(C3) of \cite{GRZ_MgEx09}. As mentioned before, all properties for the constant diffusion coefficient $B \equiv G$ follow immediately. Concerning (C1), let $y_n \to y$ in $L^2_{\sigma}$ and $z \in H^3$. By Lemma \ref{Lem_ext_F-alpha} we have
\begin{align*}
|\langle F_{\alpha}(y_n)-F_{\alpha}(y), z \rangle _{(-3,3)}| \leq |\langle y_n\otimes y_n -y\otimes y, \nabla z\rangle_{L^2}|+|\langle y_n-y, (-\Delta)^{\alpha}z\rangle_{L^2}|.
\end{align*}
Clearly, the second summand converges to $0$ as $n \longrightarrow \infty$ and, furthermore,
$$ |\langle y_n\otimes y_n -y\otimes y,\nabla z\rangle_{L^2}| \leq ||\nabla z||_{L^{\infty}}||y_n\otimes y_n -y \otimes y||_{L^1} \leq ||\nabla z||_{L^{\infty}}\big(||y_n||_{L^2}+||y||_{L^2}\big)\cdot ||y_n-y||_{L^2} \underset{n \to \infty}{\longrightarrow}0,
$$	which yields the required demicontinuity of $F_{\alpha}$.
Next, for $z \in H^3$, due to $\text{div}(z) = 0$, \eqref{eq_frac_Laplace_pairing_aux2} and the estimate $(1+|k|^{2 \alpha}) \geq 2^{-\alpha} (1 + |k|^{2})^{\alpha}$, we find
$$\langle F_{\alpha}(z),z \rangle_{(-3,3)} = -||(-\Delta)^{\frac{\alpha}{2}}z||_{L^2}^2 \leq - C \mathcal{N}_1(z)+||z||^2_{L^2},$$which gives the required coercivity (C2).
Finally, for $y \in L^2_{\sigma}$ we bound $F_{\alpha}(y)$ in $H^{-3}$ via
\begin{align}
\notag ||F_{\alpha}(y)||_{H^{-3}} &\leq ||\text{div}(y\otimes y)||_{H^{-3}}+||(-\Delta)^{\alpha}y||_{H^{-3}} \\&\notag\leq
\underset{\substack{z \in H^3,\\ ||z||_{H^{3}} \leq 1}}{\text{sup}}|\langle y\otimes y, \nabla z\rangle_{L^2}|+ \underset{\substack{z \in H^3,\\ ||z||_{H^{3}} \leq 1}}{\text{sup}}|\langle y,(-\Delta)^{\alpha}z\rangle_{L^2}| \\&\notag\leq\bigg(
\underset{\substack{z \in H^3,\\ ||z||_{H^3} \leq 1}}{\text{sup}}||\nabla z||_{L^{\infty}}\bigg)||y||^2_{L^2}+ \bigg(\underset{\substack{z \in H^3,\\ ||z||_{H^3} \leq 1}}{\text{sup}}||(-\Delta)^{\alpha}z||_{L^2}\bigg) ||y||_{L^2}
\\&\label{eq_bounddness_F-alpha} \leq C\big(1+||y||_{L^2}^2\big),
\end{align} where we used Lemma \ref{Lem_ext_F-alpha} and the constant $C$, which is independent of $y$ and $z$, comes from the Sobolev embeddings $H^2 \hookrightarrow L^{\infty}$ and $H^{3} \hookrightarrow H^{2\alpha}$. Consequently, still denoting the possibly changing constant by $C$,
$$||F_{\alpha}(y)||^2_{H^{-3}} \leq C(1+||y||_{L^2}^2)^2\leq C(1+||y||^4_{L^2}),$$which gives the desired growth condition. Hence all prerequisites in order to apply Theorem 4.7. of \cite{GRZ_MgEx09} are fulfilled and \ref{Thm.3.1-analog_itm1} follows.
\\
\\
\textit{Proof of \ref{Thm.3.1-analog_itm2}:} Since \cite[Lemma A.1.]{HZZ_stoNSE} applies to our setting, it follows as in the proof of Theorem 3.1. in \cite{HZZ_stoNSE} that $(P_n)_{n\geq 1}$ as in the assertion is tight in $\mathcal{P}(\mathbb{S})$ with $\mathbb{S}:= C_{\mathrm{loc}}(\mathbb{R}_+,H^{-3})\cap L^2_{\mathrm{loc}}(\mathbb{R}_+,L^2_{\sigma})$. Without loss of generality, we assume $P_n \underset{n \to \infty}{\longrightarrow} P$ for some $P \in \mathcal{P}(\mathbb{S})$. Since for any $t >0$ the map $x  \mapsto \underset{r \in [0,t]}{\text{sup}}||x(r)||^2_{L^2}$ is lower semicontinuous on $C(\mathbb{R}_+,H^{-3})$, we may assume $P\big(L^{\infty}_{\mathrm{loc}}(\mathbb{R}_+,L^2_{\sigma})\big) = 1$. From here, the proof can be concluded exactly as in \cite{HZZ_stoNSE}. 
\qed

\section{Beltrami waves}\label{App_Beltrami_waves} The material of this appendix is taken from \cite[Section 5]{BV19b}.
Let $\Lambda \subseteq \mathbb{S}^2 \cap \mathbb{Q}^3$ be finite such that $\Lambda = -\Lambda$. For $\xi \in \Lambda$ let $A_{\xi} \in \mathbb{S}^2 \cap \mathbb{Q}^3$ such that
$$A_{\xi} \cdot \xi = 0,\quad A_{\xi} = A_{-\xi}$$
and define the complex vector
$$B_{\xi} := \frac{1}{\sqrt{2}}\big(A_{\xi}+i\xi \times A_{\xi}\big).$$
By construction, $B_{\xi} \in \mathbb{C}^3$ has the properties
$$|B_{\xi}| = 1, \quad B_{\xi}\cdot \xi = 0, \quad i\xi \times B_{\xi} = B_{\xi}, \quad B_{\xi} = \overline{B_{-\xi}}.$$
Hence, if $\lambda \in \mathbb{Z}$ is such that $\lambda\xi \in \mathbb{Z}^3$, a direct calculation shows that for each $\xi \in \Lambda$ the vector field
$$W_{\xi}(x) := W_{\xi,\lambda}(x) := B_{\xi}e^{i\lambda \xi \cdot x}$$
is $\mathbb{T}^3$-periodic, divergence-free and an eigenfunction of the $\curl$-operator with eigenvalue $\lambda$. Such vector fields are called complex Beltrami waves and are particularly useful due to the following two results (cf. \cite[Proposition 5.5, Proposition 5.6]{BV19b}).
\begin{proposition}\label{Prop_Beltrami_waves}
	Let $\Lambda$ and $\lambda$ be as above and $a_{\xi} \in \mathbb{C}$, $\xi \in \Lambda$, a family of coefficients such that $a_{-\xi} = \overline{a_{\xi}}$. Then the vector field
	$$W(x) := \sum_{\xi \in \Lambda}a_{\xi}B_{\xi}e^{i\lambda \xi \cdot x}$$
	is $\mathbb{R}^3$-valued and divergence-free with $\curl W = \lambda W$. Hence, it is a stationary solution to the Euler equations
	\begin{equation}\label{Beltrami_sol_EE}
	\divv \big(W \otimes W\big) = \nabla \frac{|W|^2}{2}.
	\end{equation}
	Furthermore, we have for all $\xi, \xi' \in \Lambda$
	\begin{equation}\label{Beltrami_aux_1}
	B_{\xi}\otimes B_{-\xi} + B_{-\xi}\otimes B_{\xi} = \Id-\xi \otimes \xi
	\end{equation}
	and
	\begin{equation}\label{Beltrami_aux_2}
	\divv \big(W_{\xi}\otimes W_{\xi'}+ W_{\xi'}\otimes W_{\xi}\big) = \nabla\big(W_{\xi}\cdot W_{\xi'}\big).
	\end{equation}
\end{proposition}
The following geometric lemma is the reason why Beltrami waves are used in the iterative scheme of Section 3 to obtain a cancellation for the oscillation error. Below, for a symmetric $3\times 3$-matrix $A$ we denote the ball of radius $r>0$ centered at $A$ in the space of symmetric $3\times 3$ matrices by $B_r(A)$.
\begin{lemma}\label{Lemma_geom}
	There is a small $r_0 >0$ such that there exist pairwise disjoint finite sets $\Lambda_{\alpha} \subseteq \mathbb{S}^2 \cap \mathbb{Q}^3$, $\alpha \in \{0,1\}$, with the same cardinality and smooth positive functions $\gamma^{(\alpha)}_{\xi} \in C^{\infty}\big(\overline{B_{r_0}(\Id)}\big)$ with the following properties. For $\alpha \in \{0,1\}$, we have $\Lambda_{\alpha} = -\Lambda_{\alpha}$ and $\gamma^{(\alpha)}_{\xi} = \gamma^{(\alpha)}_{-\xi}$ for each $\xi \in \Lambda_{\alpha}$. Moreover, for each $R \in \overline{B_{r_0}(\Id)}$, we have the identity
	\begin{equation}\label{eq_geom_lemma}
	R = \frac{1}{2}\sum_{\xi \in \Lambda_{\alpha}}\bigg(\gamma^{(\alpha)}_{\xi}(R)\bigg)^2(\Id-\xi \otimes \xi).
	\end{equation}
\end{lemma}
It is useful to denote by $n_0$ the smallest natural number such that $n_0\Lambda_{\alpha}\subseteq \mathbb{Z}$ for $ \alpha \in \{0,1\}$. 
\section{Stationary phase lemma}\label{AppD}
In estimating the various terms of the Reynolds stress, we will often employ the following result, cf. \cite[Proposition G.1]{BDLIS15}.
\begin{proposition} \label{lemma_stationary_phase}
    Let $\xi \in \mathbb{S}^{2}$ and $\lambda \in \N$ be fixed. For a smooth vector field $a \in C^{\infty}(\mbT^{3};\R^{3})$, let $F(x) := a(x) e^{i\lambda \xi \cdot x}$. Then we have for any $\varepsilon \in (0,1)$ and $N \in \mathbb{N}$
    \begin{align*}
        \| \mcR (F) \|_{C^{\varepsilon}} \lesssim \frac{\| a \|_{C^{0}}}{\lambda^{1-\varepsilon}} + \frac{[a]_{C^{N}}}{\lambda^{N-\varepsilon}} + \frac{[a]_{C^{N+\varepsilon}}}{\lambda^N},
    \end{align*}
where the implicit constant depends only on $\varepsilon$ and $N$.
\end{proposition}

\section*{Acknowledgements}
M.R. and A.S. gratefully acknowledge the support by the German Research Foundation (DFG) through the IRTG 2235 and the SFB 1283, respectively. The authors would further like to thank Martina Hofmanov\'{a} for helpful discussions and the anonymous referees for their helpful comments that have improved the presentation of the article.

\printbibliography

 \end{document}